\documentclass[12pt,reqno]{amsart}
\usepackage{amsmath,amssymb,extarrows}
\usepackage{url}
\usepackage{tikz,enumerate}
\usepackage{diagbox}
\usepackage{appendix}
\usepackage{epic}

\usepackage{float} 

\usepackage{cite}

\usepackage{hyperref}
\usepackage{array}

\usepackage{booktabs}

\allowdisplaybreaks[4]

\newtheorem{theorem}{Theorem}

\newtheorem{lemma}[theorem]{Lemma}

\theoremstyle{definition}

\theoremstyle{remark}
\newtheorem{rem}{Remark}
\numberwithin{equation}{section}
\numberwithin{theorem}{section}
\numberwithin{defn}{section}

\begin{document}
\title[Some New Modular Rank Three Nahm Sums]
 {Some New Modular Rank Three Nahm Sums from a Lift-Dual Operation}

\author{Zhineng Cao and Liuquan Wang}

\address[Z.\ Cao]{School of Mathematics and Statistics, Wuhan University, Wuhan 430072, Hubei, People's Republic of China}
\email{zhncao@whu.edu.cn}

\address[L.\ Wang]{School of Mathematics and Statistics, Wuhan University, Wuhan 430072, Hubei, People's Republic of China}

\email{wanglq@whu.edu.cn;mathlqwang@163.com}




\begin{abstract}
Around 2007, Zagier discovered some rank two and rank three Nahm sums, and their modularity have now all been confirmed. Zagier also observed that the dual of a modular Nahm sum is likely to be modular.  This duality observation motivates us to discover some new modular rank three Nahm sums by a lift-dual operation. We first lift Zagier's rank two Nahm sums to rank three and then calculate their dual, and we show that these dual Nahm sums are indeed modular. We achieve this by establishing the corresponding Rogers--Ramanujan type identities, which express these Nahm sums as modular infinite products.
\end{abstract}

\maketitle

\section{Introduction}

As an important problem linking the theory of $q$-series and modular forms, Nahm's problem is to determine all positive definite matrix $A\in \mathbb{Q}^{r\times r}$, $r$-dimensional column vector $B\in \mathbb{Q}^r$ and rational scalar $C$ such that the Nahm sum
    \begin{align}\label{eq-Nahm}
    f_{A,B,C}(q):=\sum_{n=(n_1,\dots,n_r)^\mathrm{T} \in \mathbb{N}^r}\frac{q^{\frac{1}{2}n^\mathrm{T}An+n^\mathrm{T}B+C}}{(q;q)_{n_1} \cdots (q;q)_{n_r}}
\end{align}
is modular. Here and below we use $q$-series notations: for $n\in \mathbb{N}\cup \{\infty\}$ we define
\begin{align}
(a;q)_n&:=\prod_{k=0}^{n-1} (1-aq^k), \\
(a_1,\dots,a_m;q)_n&:=(a_1;q)_n\cdots (a_m;q)_n.
\end{align}
Modular Nahm sums usually appear as characters of some rational conformal field theories. A famous example arises from the Rogers--Ramanujan identities \cite{Rogers}:
\begin{align}
    \sum_{n=0}^\infty\frac{q^{n^2}}{(q;q)_n}
    =
    \frac{1}{(q,q^4;q^5)_\infty}, \quad
    \sum_{n=0}^\infty\frac{q^{n^2+n}}{(q;q)_n}
    =
    \frac{1}{(q^2,q^3;q^5)_\infty}.\label{RR}
\end{align}
They imply that the Nahm sums $f_{2,0,-1/60}(q)$ and $f_{2,1,11/60}(q)$ are modular, and they correspond to two characters of the Lee--Yang model (see e.g.\ \cite{Kac}).

For convenience, when the Nahm sum $f_{A,B,C}(q)$ is modular, we call $(A,B,C)$ as a modular triple. Nahm's conjecture, sated explicitly by Zagier \cite{Zagier}, provides a criterion on the matrix $A$ so that it becomes the matrix part of a modular triple. This conjecture has been confirmed in the rank one case by Zagier \cite{Zagier}. It does not hold for a general rank since Vlasenko and Zwegers \cite{VZ} found that the matrices
\begin{align}\label{matrix-VZ}
    A= \begin{pmatrix}
        3/4 & -1/4 \\ -1/4 & 3/4
    \end{pmatrix}, \begin{pmatrix}
        3/2 & 1/2 \\ 1/2 & 3/2
    \end{pmatrix}
\end{align}
do not satisfy Nahm's criterion but do appear as the matrix part of some modular triples.
Recently, Calegari, Garoufalidis and Zagier \cite{CGZ} proved that one direction of Nahm's conjecture is true.

When the rank $r\geq 2$, Nahm's problem is far from being solved. One way to tackle this problem is to provide as many modular triples as possible, and the problem is solved when the list of modular triples is complete.  In the rank two and three cases, after an extensive search, Zagier \cite[Table 2]{Zagier} provided 11 and 12 sets of possible modular Nahm sums, respectively. Their modularity have now all been confirmed by the works of Vlasenko--Zwegers \cite{VZ}, Cherednik--Feigin \cite{Feigin}, Cao--Rosengren--Wang \cite{CRW} and Wang \cite{Wang-rank2,Wang-rank3}.

Zagier \cite[p.\ 50, (f)]{Zagier} observed that there might exist some dual structure among modular triples. For a modular triple $(A,B,C)$, we define its dual as the image of the operator:
\begin{align}
    \mathcal{D}:(A,B,C)\longmapsto (A^\star, B^\star, C^\star)=(A^{-1},A^{-1}B,\frac{1}{2}B^\mathrm{T} A^{-1}B-\frac{r}{24}-C).
\end{align}
Zagier conjectured that $\mathcal{D}(A,B,C)$ is still a modular triple. Recently, Wang \cite{Wang2024} presented some counterexamples involving rank four Nahm sums of this conjecture.

This work aims to provide more modular Nahm sums. The idea of constructing new modular triples consists of two steps. We first lift some known rank two modular triples to rank three, and then we consider their duals.  For any
\begin{align*}
&A=\begin{pmatrix} a_1 & a_2 \\ a_2 & a_3\end{pmatrix}, \quad B=\begin{pmatrix} b_1 \\ b_2 \end{pmatrix},
\end{align*}
we define an operator which we call as \emph{lifting operator} to lift $A$ to a $3\times 3$ matrix and $B$ to a three dimensional vector and keep the value of $C$:
\begin{align}
    \mathcal{L}: (A,B,C)\longmapsto (\widetilde{A},\widetilde{B},C)
\end{align}
where
\begin{align*}
&\widetilde{A}=\begin{pmatrix} a_1 & a_2+1 & a_1+a_2 \\ a_2+1 & a_3 & a_2+a_3 \\ a_1+a_2 & a_2+a_3 & a_1+2a_2+a_3 \end{pmatrix}, \quad \widetilde{B}=\begin{pmatrix} b_1 \\ b_2 \\ b_1+b_2\end{pmatrix}.
\end{align*}
It is known that
\begin{align}\label{eq-lift-id}
    f_{A,B,0}(q)=f_{\widetilde{A},\widetilde{B},0}(q).
\end{align}
This fact appeared first in Zwegers' unpublished work \cite{ZwegersTalk} according to Lee's thesis \cite{LeeThesis}. See \cite{LeeThesis} and \cite{CRW} for a proof.

If $(A,B,C)$ is a rank two modular triple, then from \eqref{eq-lift-id} we get a rank three  modular triple $(\widetilde{A},\widetilde{B}, C)$ for free subject to the condition that $\widetilde{A}$ is positive definite. Zagier's duality conjecture then motivates us to consider the dual example of $(\widetilde{A},\widetilde{B},C)$. That is, from a rank two modular triple $(A,B,C)$ we get a candidate of rank three modular triple $\mathcal{D}\mathcal{L}(A,B,C)$ when $\widetilde{A}$ is positive definite.  We shall call this process \emph{lift-dual operation} to $(A,B,C)$.

It should be noted that the lift-dual process does not always generate new modular triples. For example, the two matrices in  \eqref{matrix-VZ} lift to singular matrices. Therefore, they do not generate new rank three modular triples from the lift-dual operation.

The main object of this work is to apply the lifting operator to Zagier's rank two examples \cite[Table 2]{Zagier} and check if we get new modular triples. We list the lifting matrices in Table \ref{tab-lift}. It is easy to see that only four of them are positive definite. Namely,  the lift of Zagier's matrices are positive definite only for Examples 1, 3, 9 and 11.
\begin{table}[htbp]\label{tab-lift}
\renewcommand{\arraystretch}{1.9}
\begin{tabular}{cccc} \hline
  Exam.\ No.\ & Matrix $A$ & Lift $\widetilde{A}$  & $\det \widetilde{A}$   \\
  \hline
1 & $\left(\begin{smallmatrix}
a & 1-a \\ 1-a & a
\end{smallmatrix} \right)$  &  $\left(\begin{smallmatrix}
     a & 2-a & 1 \\
     2-a & a & 1 \\
     1 & 1 & 2
 \end{smallmatrix} \right)$ & $4a-4$  \\
 \hline
 2 & $\left(\begin{smallmatrix} 2 & 1 \\ 1 & 1 \end{smallmatrix}\right)$ &  $\left(\begin{smallmatrix} 2 & 2 & 3  \\ 2 & 1 & 2 \\ 3 & 2 & 5 \end{smallmatrix}\right)$ & $-3$ \\
 \hline
 3 & $\left(\begin{smallmatrix} 1 & -1 \\ -1 & 2 \end{smallmatrix}\right)$  & $\left(\begin{smallmatrix} 1 & 0 & 0  \\ 0 & 2  & 1 \\ 0 & 1 & 1 \end{smallmatrix}\right)$  & 1 \\
 \hline
 4 & $\left(\begin{smallmatrix} 4 & 1 \\ 1 & 1 \end{smallmatrix}\right)$ &  $\left(\begin{smallmatrix} 4 & 2 & 5 \\ 2 & 1
 & 2 \\ 5 & 2 & 7\end{smallmatrix}\right)$  & $-1$\\
 \hline
 5 & $\left(\begin{smallmatrix}   1/3 & -1/3 \\ -1/3 & 4/3 \end{smallmatrix} \right)$  & $\left(\begin{smallmatrix} 1/3 & 2/3 & 0  \\ 2/3 & 4/3  & 1 \\ 0 & 1 & 1 \end{smallmatrix}\right)$ & $-1/3$ \\
 \hline
 6 & $\left(\begin{smallmatrix} 4 & 2 \\ 2 & 2 \end{smallmatrix}\right)$ & $\left(\begin{smallmatrix} 4 & 3 & 6 \\ 3 & 2 & 4 \\ 6 & 4 & 10 \end{smallmatrix}\right)$ & $-2$  \\
 \hline
 7 & $\left(\begin{smallmatrix} 1/2 & -1/2 \\ -1/2 & 1 \end{smallmatrix}  \right)$ & $\left(\begin{smallmatrix} 1/2 & 1/2  & 0 \\ 1/2 & 1 & 1/2 \\ 0 & 1/2 & 1/2 \end{smallmatrix}\right)$ & 0 \\
 \hline
 8 & $\left( \begin{smallmatrix} 3/2 & 1 \\ 1 & 2 \end{smallmatrix}  \right)$  & $\left(\begin{smallmatrix} 3/2 & 2 &5/2 \\ 2 & 2  & 3 \\ 5/2 & 3 & 11/2 \end{smallmatrix}\right)$ & $-3/2$ \\
 \hline
 9 & $\left( \begin{smallmatrix} 1 & -1/2 \\ -1/2 & 3/4 \end{smallmatrix}   \right)$  & $\left(\begin{smallmatrix} 1& 1/2 & 1/2  \\ 1/2 & 3/4 & 1/4 \\  1/2 & 1/4 & 3/4 \end{smallmatrix}\right)$ & $1/4$ \\
 \hline
 10 & $\left(\begin{smallmatrix} 4/3 & 2/3 \\ 2/3 & 4/3 \end{smallmatrix}\right)$ & $\left(\begin{smallmatrix} 4/3 & 5/3 & 2 \\ 5/3 & 4/3 & 2 \\ 2 & 2 & 4 \end{smallmatrix}\right)$ & $-4/3$  \\
 \hline
 11 & $\left(\begin{smallmatrix} 1 &-1/2\\ -1/2 & 1 \end{smallmatrix}\right)$  & $\left(\begin{smallmatrix} 1 & 1/2 & 1/2 \\ 1/2  & 1 & 1/2 \\ 1/2 & 1/2 &1  \end{smallmatrix}\right)$ & $1/2$ \\
  \hline
\end{tabular}
\\[2mm]
\caption{Matrices from Zagier's rank two examples and their lifts}
\label{tab-known}
\end{table}

The dual of the lift of the matrix in Example 3 is
\begin{align}
    \widetilde{A}^\star=\begin{pmatrix}
        1 & 0 & 0 \\
        0 & 1 & -1 \\
        0 & -1 & 2
    \end{pmatrix}.
\end{align}
Obviously, any Nahm sum for this matrix can be decomposed into the product of a rank one Nahm sum and a rank two Nahm sum, and hence is not essentially new. Therefore, we will focus on the dual of the lift of Examples 1, 9 and 11. We find that they indeed produce new modular Nahm sums.

For each of the Nahm sums we consider, we investigate their modularity by establishing  the corresponding Rogers--Ramanujan type identities. To be precise, we express the Nahm sums using the functions
\begin{align}\label{Jm}
J_m:=(q^m;q^m)_\infty \quad \text{and} \quad J_{a,m}:=(q^a,q^{m-a},q^m;q^m)_\infty.
\end{align}
Let $q=e^{2\pi i \tau}$ where $\mathrm{Im}~ \tau>0$. The functions $J_m$ and $J_{a,m}$ are closely related to the Dedekind eta function
\begin{align}\label{eta-defn}
    \eta(\tau):=q^{1/24}(q;q)_\infty
\end{align}
and the generalized Dedekind eta function
\begin{align}\label{general-eta}
    \eta_{m,a}(\tau):=q^{mB(a/m)/2}(q^a,q^{m-a};q^m)_\infty
\end{align}
where $B(x)=x^2-x+1/6$. It is well-known that $\eta(\tau)$ is a modular form of weight $1/2$ and $\eta_{m,a}(\tau)$ is a modular form of weight zero. The modularity of a  Nahm sum will be clear once we write it in terms of $J_m$ and $J_{a,m}$.

We shall use an example to briefly illustrate our work. Zagier's Example 11 asserts that $(A,B_i,C_i)$ are modular triples where
\begin{equation}
\begin{split}
&A=\begin{pmatrix}
    1 & -1/2 \\  -1/2 & 1
\end{pmatrix}, ~~ B_1=\begin{pmatrix}
-1/2 \\ 0 \end{pmatrix}, ~~ B_2=\begin{pmatrix}
    0 \\ -1/2
\end{pmatrix}, ~~ B_3=\begin{pmatrix}
    0 \\ 0
\end{pmatrix}, \\
&C_1=1/20, \quad C_2=1/20, \quad C_3=-1/20.
\end{split}
\end{equation}
This lifts to the  modular triples $(\widetilde{A},\widetilde{B},C_i)$ where $C_i$ is as above and
\begin{align}
\widetilde{A}=\begin{pmatrix} 1 & 1/2 & 1/2 \\  1/2 & 1 & 1/2 \\ 1/2 & 1/2 & 1 \end{pmatrix}, ~~\widetilde{B}_1 =
\begin{pmatrix} -1/2 \\ 0 \\ -1/2 \end{pmatrix},  ~~\widetilde{B}_2= \begin{pmatrix} 0 \\ -1/2 \\ -1/2 \end{pmatrix}, ~~ \widetilde{B}_3=\begin{pmatrix} 0 \\ 0 \\ 0 \end{pmatrix}.
\end{align}
We may also include the vector $(-1/2,-1/2,0)^\mathrm{T}$ since $n_1,n_2,n_3$ are symmetric to each other in the quadratic form $\frac{1}{2}n^\mathrm{T}\widetilde{A}n$. Considering its dual, we expect that $(\widetilde{A}^\star,\widetilde{B}_i^\star,C_i^\star)$ ($i=1,2,3$) are modular triples where
\begin{align}
&\widetilde{A}^\star=\begin{pmatrix} 3/2 & -1/2 & -1/2 \\  -1/2 & 3/2 & -1/2 \\ -1/2 & -1/2 & 3/2 \end{pmatrix}, ~~ \widetilde{B}_1^\star=
\begin{pmatrix} -1/2 \\ 1/2 \\ -1/2 \end{pmatrix}, ~~\widetilde{B}_2^\star=
\begin{pmatrix} 1/2 \\ -1/2 \\ -1/2 \end{pmatrix},  ~~\widetilde{B}_3^\star=\begin{pmatrix} 0 \\ 0 \\ 0 \end{pmatrix}, \nonumber  \\
&C_1^\star=-3/40, \quad C_2^\star=-3/40, \quad C_3^\star=3/40.
\end{align}
We can also include the vector $(-1/2,-1/2,1/2)^\mathrm{T}$. Due to the symmetry of the quadratic form generated by $\widetilde{A}^\star$, there are essentially only two different Nahm sums to consider. We establish the following identities to confirm their modularity.
\begin{theorem}\label{thm-lift-11}
We have
\begin{align}
  & \sum_{i,j,k\geq 0} \frac{q^{3i^2+3j^2+3k^2-2ij-2ik-2jk}}{(q^4;q^4)_i(q^4;q^4)_j(q^4;q^4)_k}=\frac{J_6^5J_{28,60}}{J_3^2J_4^2J_{12}^2}+2q^3\frac{J_2^2J_3J_{12}J_{12,60}}{J_1J_4^3J_6}-q^4\frac{J_6^5J_{8,60}}{J_3^2J_4^2J_{12}^2}, \label{eq-thm-11-1} \\
  & \sum_{i,j,k\geq 0} \frac{q^{3i^2+3j^2+3k^2-2ij-2ik-2jk-2i+2j-2k}}{(q^4;q^4)_i(q^4;q^4)_j(q^4;q^4)_k}=2\frac{J_2^2J_3J_{12}J_{24,60}}{J_1J_4^3J_6}+q\frac{J_6^5J_{16,60}}{J_3^2J_4^2J_{12}^2}+q^5\frac{J_6^5J_{4,60}}{J_3^2J_4^2J_{12}^2}. \label{eq-thm-11-2}
\end{align}
\end{theorem}

The main difficulty of this work lies in finding and proving the Rogers--Ramanujan type identities for Nahm sums. We achieve this by applying various $q$-series techniques including the constant term method, integral method and Bailey pairs. There are some unexpected things during our proof. For instance, in order to prove Theorem \ref{thm-lift-11}, we need to use two identities found by Vlasenko--Zwegers \cite{VZ} on Zagier's Example 10.

The rest of this paper is organized as follows. In Section \ref{sec-pre} we review some known $q$-series identities which will be used in our proof. We also recall Bailey's lemma and its consequences. Sections \ref{sec-exam1}--\ref{sec-exam11} are devoted to discussing the new modular triples we found after applying the lift-dual operation to Zagier's Examples 1, 9 and 11, respectively. Finally, in Section \ref{sec-applictaion}, we discuss a special case of the matrix with parameters in Section \ref{sec-exam1}, and we show that there are other modular triples in addition to those described in Section \ref{sec-exam1}.

\section{Preliminaries}\label{sec-pre}

We need the Jacobi triple product identity (see e.g. \cite[Theorem 2.8]{Andrews-book}):
\begin{align}\label{JTP}
(q,z,q/z;q)_\infty=\sum_{n=-\infty}^\infty (-1)^nq^{\binom{n}{2}}z^n.
\end{align}
For convenience, besides \eqref{Jm} we also use the notation:
\begin{align}
\overline{J}_{a,m}=(-q^a,-q^{m-a},q^m;q^m)_\infty=\frac{J_m^2J_{2a,2m}}{J_{a,m}J_{2m}}.
\end{align}

 Whenever we write a $q$-series $f(q)$ in terms of $J_m$ and $J_{a,m}$ defined in \eqref{Jm}, we can use \eqref{eta-defn} and \eqref{general-eta} to find suitable $C$ so that $q^Cf(q)$ is modular. In the case of Nahm sums, if we are able to write a Nahm sum $f_{A,B,0}(q)$ as
\begin{align}
    f_{A,B,0}(q)=f_1(q)+f_2(q)+\cdots+f_k(q)
\end{align}
where $q^{C}f_i(q)$ are modular of weight zero for $i=1,2,\dots,k$, then we see that the Nahm sum $f_{A,B,C}(q)=q^Cf_{A,B,0}(q)$ is modular.

The basic hypergeometric series ${}_r\phi_s$ is defined as:
$${}_r\phi_s\bigg(\genfrac{}{}{0pt}{} {a_1,\dots,a_r}{b_1,\dots,b_s};q,z  \bigg):=\sum_{n=0}^\infty \frac{(a_1,\dots,a_r;q)_n}{(q,b_1,\dots,b_s;q)_n}\Big((-1)^nq^{\binom{n}{2}} \Big)^{1+s-r}z^n.$$
We now list some formulas for such series which will be invoked in our proofs.
\begin{enumerate}[(1)]
\item The $q$-binomial theorem \cite[Theorem 2.1]{Andrews-book} asserts that
\begin{align}\label{q-binomial}
\sum_{n\geq 0} \frac{(a;q)_n}{(q;q)_n}z^n=\frac{(az;q)_\infty}{(z;q)_\infty}, \quad |z|<1.
\end{align}
\item As corollaries to \eqref{q-binomial}, we have Euler's $q$-exponential identities \cite[Corollary 2.2]{Andrews-book}:
\begin{align}\label{Euler1}
\sum_{n\geq 0} \frac{z^n}{(q;q)_n}=\frac{1}{(z;q)_\infty}, \quad
\sum_{n\geq 0} \frac{z^nq^{\frac{n^2-n}{2}}}{(q;q)_n}=(-z;q)_\infty.
\end{align}
Here we need $|z|<1$ for the first identity.

\item The $q$-Gauss summation formula \cite[(1.5.1)]{Gasper-Rahman}:
\begin{align} \label{Gauss}
{}_2\phi_1\bigg(\genfrac{}{}{0pt}{} {a,b}{c};q,c/ab  \bigg)=\frac{(c/a,c/b;q)_\infty}{(c,c/ab;q)_\infty},  \quad \left| \frac{c}{ab} \right|<1.
\end{align}

\item The sum of a ${}_1\phi_1$ series \cite[\uppercase\expandafter{\romannumeral2}.5]{Gasper-Rahman}:
\begin{align} \label{1phi1}
{}_1\phi_1\bigg(\genfrac{}{}{0pt}{} {a}{c};q,c/a  \bigg)=\frac{(c/a;q)_\infty}{(c;q)_\infty}.
\end{align}

\item A $q$-analogue of Bailey's ${}_2F_{1}(-1)$ sum \cite[\uppercase\expandafter{\romannumeral2}.10]{Gasper-Rahman}:
\begin{align} \label{Bailey's}
{}_2\phi_2\bigg(\genfrac{}{}{0pt}{} {a,q/a}{-q,b};q,-b  \bigg)=\frac{(ab,bq/a;q^2)_\infty}{(b;q)_\infty}.
\end{align}

\item A $q$-analogue of Gauss' ${}_2F_{1}(-1)$ sum \cite[\uppercase\expandafter{\romannumeral2}.11]{Gasper-Rahman}:
\begin{align} \label{Gauss'}
{}_2\phi_2\bigg(\genfrac{}{}{0pt}{} {a^2,b^2}{abq^{1/2},-abq^{1/2}};q,-q  \bigg)=\frac{(a^2q,b^2q;q^2)_\infty}{(q,a^2b^2q;q^2)_\infty}.
\end{align}
\end{enumerate}

One main step in our calculations of rank three Nahm sums is to reduce them to single or double sums and then employ some known identities. Here we recall some single sum Rogers--Ramanujan type identities:
\begin{align}
&\sum_{n\geq 0} \frac{q^{n^2}(-q;q^2)_n}{(q^4;q^4)_{n}}=\frac{J_2J_{3,6}}{J_{1}J_{4}}, \quad \text{(S. 25)} \label{S. 25}\\
&\sum_{n\geq 0} \frac{q^{(n^2+n)/2}}{(q;q)_n(q;q^2)_{n+1}}=\frac{J_2J_{14}^{3}}{J_1J_{1,14}J_{4,14}J_{6,14}}, \quad \text{(S. 80)} \label{S. 80}\\
&\sum_{n\geq 0} \frac{q^{(n^2+n)/2}}{(q;q)_n(q;q^2)_{n}}=\frac{J_2J_{14}^{3}}{J_1J_{2,14}J_{3,14}J_{4,14}}, \quad \text{(S. 81)} \label{S. 81}\\
&\sum_{n\geq 0} \frac{q^{(n^2+3n)/2}}{(q;q)_n(q;q^2)_{n+1}}=\frac{J_2J_{14}^{3}}{J_1J_{2,14}J_{5,14}J_{6,14}}, \quad \text{(S. 82)} \label{S. 82}\\
&\sum_{n\geq 0} \frac{q^{n^2}}{(q;q^2)_n(q^4;q^4)_{n}}=\frac{J_2J_{14}J_{3,28}J_{11,28}}{J_1J_{28}J_{4,28}J_{12,28}}, \quad \text{(S. 117)} \label{S. 117}\\
&\sum_{n\geq 0} \frac{q^{n^2+2n}}{(q;q^2)_n(q^4;q^4)_{n}}=\frac{J_2J_{1,14}J_{12,28}}{J_1J_{4}J_{28}}, \quad \text{(S. 118)} \label{S. 118}\\
&\sum_{n\geq 0} \frac{q^{n^2+2n}}{(q;q)_{2n+1}(-q^2;q^2)_{n}}=\frac{J_2J_{5,14}J_{4,28}}{J_1J_{4}J_{28}}, \quad \text{(S. 119)} \label{S. 119}\\
&\sum_{n\geq 0} \frac{(-1)^nq^{n^2}}{(q^4;q^4)_{n}(-q;q^2)_{n}}=\frac{J_{1,14}J_{5,14}J_{7}}{J_{2,14}J_{4,14}J_{14}}, \quad \text{(\cite[Entry 3.5.4]{Andrews-Berndt})} \label{Entry 3.5.4}\\
&\sum_{n\geq 0} \frac{(-1)^nq^{n^2+2n}}{(q^4;q^4)_{n}(-q;q^2)_{n}}=\frac{J_{3,14}J_{5,14}J_{7}}{J_{4,14}J_{6,14}J_{14}}, \quad \text{(\cite[Entry 3.5.5]{Andrews-Berndt})} \label{Entry 3.5.5}\\
&\sum_{n\geq 0} \frac{(-1)^nq^{n^2+2n}}{(q^4;q^4)_{n}(-q;q^2)_{n+1}}=\frac{J_{1,14}J_{3,14}J_{7}}{J_{2,14}J_{6,14}J_{14}}, \quad \text{(\cite[Entry 3.5.6]{Andrews-Berndt})} \label{Entry 3.5.6}\\
&\sum_{n\geq 0} \frac{q^{n^2}(-1;q)_{n}}{(q;q)_{n}(q;q^2)_{n}}
=\sum_{n\geq 0} \frac{q^{n^2}(-q;q)_{n}}{(q;q)_{n}(q;q^2)_{n+1}}
=\frac{J_{2}J_{3}^2}{J_{1}^2J_{6}}, \label{Entry 4.2.8+4.2.9} \\
&\qquad \qquad \qquad \text{(\cite[Entries 4.2.8 and 4.2.9]{Andrews-Berndt})}  \nonumber  \\
&\sum_{n\geq 0} \frac{(-1)^nq^{n^2+2n}(q;q^2)_{n}}{(q^4;q^4)_{n}}=\frac{J_{1}J_{6}^2J_{2,12}}{J_{2}^2J_{1,6}J_{12}}, \quad \text{(\cite[Entry 4.2.11]{Andrews-Berndt})} \label{Entry 4.2.11} \\
&\sum_{n\geq 0} \frac{q^{2n^2}(-aq,-q/a;q^2)_{n}}{(q^2;q^2)_{2n}}=\frac{(-aq^3,-q^3/a,q^6;q^6)_\infty}{(q^2;q^2)_\infty}, \quad \text{(\cite[Entry 5.3.1]{Andrews-Berndt})} \label{Entry 5.3.1} \\
&\sum_{n\geq 0} \frac{(-x;q)_{n+1}(-q/x,\rho_{1},\rho_{2};q)_{n}}{(q;q)_{2n+1}}\Big(\frac{q^2}{\rho_{1}\rho_{2}}\Big)^n =\frac{(q^2/\rho_{1},q^2/\rho_{2};q)_{\infty}}{(q,q^2/\rho_{1}\rho_{2};q)_{\infty}} \nonumber \\
&\qquad \times \sum_{n\geq 0} \frac{(\rho_{1},\rho_{2};q)_{n}}{(q^2/\rho_{1},q^2/\rho_{2};q)_{n}}(x^{n+1}+x^{-n})\Big(\frac{q^2}{\rho_{1}\rho_{2}}\Big)^nq^{(n^2+n)/2}. \quad \text{(\cite[(5.2.4)]{Andrews-Berndt})} \label{Part2-5.2.4}
\end{align}
Here we use the label (S.\ $n$) to denote the equation $(n)$ in Slater's list \cite{Slater}.

Given any series $f(z)=\sum_{n\in \mathbb{Z}} a(n)z^n$, we define the constant term extractor (with respect to $z$) $\mathrm{CT}_z$ as
\begin{align}
\mathrm{CT}_z f(z)=a(0).
\end{align}
Clearly, when the integral of $f(z)$ along a positively oriented simple closed contour around the origin is well-defined we have
\begin{align}
    \mathrm{CT} f(z)=\oint f(z) \frac{dz}{2\pi iz}.
\end{align}

We recall a useful formula from \cite{Gasper-Rahman} to evaluate contour integral of infinite products. Suppose that
$$P(z):=\frac{(a_1 z,\dots,a_A z,b_1/z,\dots,b_B/ z;q)_{\infty}}
{(c_1 z,\dots,c_C z,d_1/z,\dots,d_D/z;q)_{\infty}}$$
has only simple poles. We have \cite[Eq. (4.10.5)]{Gasper-Rahman}
\begin{align}\label{Eq. (4.10.5)}
&\oint P(z)\frac{dz}{2\pi iz}
=\frac{(a_1d_1,\dots,a_Ad_1,b_1/d_1,\dots,b_B/d_1;q)_{\infty}}
{(c_1d_1,\dots,c_Cd_1,d_2/d_1,\dots,d_D/d_1;q)_{\infty}} \nonumber\\
&\quad \times \sum_{n\geq 0}^{\infty}\frac{(c_1d_1,\dots,c_Cd_1,qd_1/b_1,\dots,qd_1/b_B;q)_{n}}
{(q,a_1d_1,\dots,a_Ad_1,qd_1/d_2,\dots,qd_1/d_D;q)_{n}}   (-d_1q^{(n+1)/2})^{n(D-B)}(\frac{b_1\cdots b_B}{d_1\cdots d_D})^n \nonumber\\
&\qquad  +\mathrm{idem}(d_1;d_2,\dots,d_D)
\end{align}
when $D>B$ or if $D=B$ and
$$\left| \frac{b_1\cdots b_B}{d_1\cdots d_D}\right|<1.$$
Here the symbol $\mathrm{idem}(d_1;d_2,\dots,d_D)$ after an expression stands for the sum of the $(D-1)$ expressions obtained from the preceding expression by interchanging $d_1$ with each $d_k$, $k=2,3,\dots,D$, and the integration is over a positively oriented simple contour so that
\begin{enumerate}[(1)]
    \item the poles of $1/(c_1 z,\cdots,c_C z;q)_{\infty}$ lie outside the contour;
    \item the origin and poles of $1/(d_1 z,\cdots,d_D z;q)_{\infty}$ lie inside the contour.
\end{enumerate}

We will also need the Bailey lemma (see e.g. \cite{BIS,Lovejoy2004,McLaughlin}). A pair of sequences $(\alpha_n(a;q),\beta_n(a;q))$ is called a Bailey pair relative to $a$ if for all $n\geq 0$,
 \begin{align}\label{defn-BP}
     \beta_n(a;q)=\sum_{k=0}^n\frac{\alpha_k(a;q)}{(q;q)_{n-k}(aq;q)_{n+k}}.
 \end{align}

\begin{lemma}[Bailey's lemma]
Suppose that $(\alpha_{n}(a;q),\beta_{n}(a;q))$ is a Bailey pair relative to $a$. Then $(\alpha_{n}'(a;q),\beta_{n}'(a;q))$ is another Bailey pair relative to $a$, where
\begin{align}\label{Bailey's lemma}
&\alpha_{n}'(a;q)=\frac{(\rho_{1},\rho_{2};q)_{n}}{(aq/\rho_{1},aq/\rho_{2};q)_{n}}\left( \frac{aq}{\rho_{1}\rho_{2}}\right)^{n}\alpha_{n}(a;q),  \\
&\beta_{n}'(a;q)=\sum_{k=0}^{n}\frac{(\rho_{1},\rho_{2};q)_{k}(aq/\rho_{1}\rho_{2};q)_{n-k}}{(aq/\rho_{1},aq/\rho_{2};q)_{n}(q;q)_{n-k}}\left( \frac{aq}{\rho_{1}\rho_{2}}\right)^{k}\beta_{k}(a;q).
\end{align}
Equivalently, if $(\alpha_n(a;q),\beta_n(a;q))$ is a Bailey pair, then
\begin{align}\label{eq-Bailey-general-id}
&\frac{1}{(aq/\rho_1,aq/\rho_2;q)_n}\sum_{j=0}^n \frac{(\rho_1,\rho_2;q)_j(aq/\rho_1\rho_2;q)_{n-j}}{(q;q)_{n-j}}\Big(\frac{aq}{\rho_1\rho_2} \Big)^j\beta_j(a;q) \nonumber \\
&=\sum_{r=0}^n \frac{(\rho_1,\rho_2;q)_r}{(q;q)_{n-r}(aq;q)_{n+r}(aq/\rho_1,aq/\rho_2;q)_r}\Big(\frac{aq}{\rho_1\rho_2} \Big)^r \alpha_r(a;q).
\end{align}
\end{lemma}
We need the following special consequences of Bailey's lemma.
\begin{enumerate}[(1)]
\item Letting $\rho_1,\rho_2\rightarrow \infty$, we obtain the Bailey pair \cite[Eq.\ (S1)]{BIS}:
\begin{align}\label{BP-S1}
\alpha_n'(a;q)=a^nq^{n^2}\alpha_n(a;q), \quad \beta_n'(a;q)=\sum_{r=0}^n \frac{a^rq^{r^2}}{(q;q)_{n-r}}\beta_r(a;q).
\end{align}
If we further let $n\rightarrow \infty$, then we deduce from \eqref{eq-Bailey-general-id} that
\begin{align}\label{eq-BP-id-key}
\sum_{n=0}^\infty a^nq^{n^2}\beta_n(a;q)=\frac{1}{(aq;q)_\infty} \sum_{n=0}^\infty a^n q^{n^2}\alpha_n(a;q).
\end{align}

\item Letting $a=1$, $q \rightarrow q^4$, $\rho_{1}=-q^2/w$,
$\rho_{2} \rightarrow \infty$, we obtain the Bailey pair
\begin{align}
&\alpha_{n}'(1;q^4)=\frac{(-q^2/w;q^4)_{n}}{(-q^2w;q^4)_{n}}w^nq^{2n^2}\alpha_{n}(1;q^4), \label{Bailey's lemma-1.1} \\
&\beta_{n}'(1;q^4)=\sum_{k=0}^{n}\frac{(-q^2/w;q^4)_{k}w^kq^{2k^2}}{(-q^2w;q^4)_{n}(q^4;q^4)_{n-k}}\beta_{k}(1;q^4).\label{Bailey's lemma-1.2}
\end{align}
If we further let $n\rightarrow \infty$, then we deduce from \eqref{eq-Bailey-general-id} that
\begin{align}\label{Bailey's lemma-1}
&\sum_{k=0}^{\infty}(-q^2/w;q^4)_{k}w^kq^{2k^2}\beta_{k}(1;q^4) \nonumber \\
&=\frac{(-q^2w;q^4)_{\infty}}{(q^4;q^4)_{\infty}}\sum_{r=0}^{\infty}\frac{(-q^2/w;q^4)_{r}w^rq^{2r^2}}{(-q^2w;q^4)_{r}}\alpha_{r}(1;q^4).
\end{align}
\item Letting $a=q^4$, $q \rightarrow q^4$, $\rho_{1}=-q^4/w$,
$\rho_{2} \rightarrow \infty$, we obtain the Bailey pair
\begin{align}
&\alpha_{n}'(q^4;q^4)=\frac{(-q^4/w;q^4)_{n}}{(-q^4w;q^4)_{n}}w^nq^{2n^2+2n}\alpha_{n}(q^4,q^4), \label{Bailey's lemma-2.1} \\
&\beta_{n}'(q^4;q^4)=\sum_{k=0}^{n}\frac{(-q^4/w;q^4)_{k}w^kq^{2k^2+2k}}{(-q^4w;q^4)_{n}(q^4;q^4)_{n-k}}\beta_{k}(q^4;q^4).\label{Bailey's lemma-2.2}
\end{align}
If we further let $n\rightarrow \infty$, then we deduce from \eqref{eq-Bailey-general-id} that
\begin{align}\label{Bailey's lemma-2}
&\frac{1}{1-q^4}\sum_{n=0}^{\infty}(-q^4/w;q^4)_{n}w^nq^{2n^2+2n}\beta_{n}(q^4;q^4) \nonumber \\
&=\frac{(-wq^4;q^4)_{\infty}}{(q^4;q^4)_{\infty}}\sum_{r=0}^{\infty}\frac{(-q^4/w;q^4)_{r}w^rq^{2r^2+2r}}{(-wq^4;q^4)_{r}}\alpha_{r}(q^4;q^4).
\end{align}
\end{enumerate}

The following result of Lovejoy \cite[p.\ 1510]{Lovejoy2004} allows us to change the parameter $a$ to $aq$ .
\begin{lemma}\label{lem-BP-lift}
If $(\alpha_n(a;q),\beta_n(a;q))$ is a Bailey pair relative to $a$, then $(\alpha_n',\beta_n')$ is a Bailey pair relative to $aq$ where
\begin{equation}
\begin{split}
\alpha_n'(aq;q)&=\frac{(1-aq^{2n+1})(aq/b;q)_n(-b)^nq^{n(n-1)/2}}{(1-aq)(bq;q)_n}\sum_{r=0}^n \frac{(b;q)_r}{(aq/b;q)_r}  \\
&\qquad \times (-b)^{-r}q^{-r(r-1)/2} \alpha_r(a;q), \\
\beta_n'(aq;q)&=\frac{(b;q)_n}{(bq;q)_n}\beta_n(a;q).
\end{split}
\end{equation}
\end{lemma}
In particular, when $b\rightarrow 0$ we obtain
\begin{equation}\label{eq-BP-lift}
\begin{split}
&\alpha_n'(aq;q)=\frac{(1-aq^{2n+1})a^nq^{n^2}}{1-aq}\sum_{r=0}^n a^{-r}q^{-r^2}\alpha_r(a;q), \\
&\beta_n'(aq;q)=\beta_n(a;q).
\end{split}
\end{equation}

The following result of Mc Laughlin \cite{McLaughlin} allows us to change the parameter $a$ to $a/q$.
\begin{lemma}\label{lem-BP-down}
If $(\alpha_n(a;q),\beta_n(a;q))$ is a Bailey pair relative to $a$, then $(\alpha_n',\beta_n')$ is a Bailey pair relative to $a/q$ where
\begin{align}
\alpha_0'(a/q;q)&=\alpha_0(a;q), \quad \alpha_n'(a/q;q)=(1-a)\Big(\frac{\alpha_n(a;q)}{1-aq^{2n}}-\frac{aq^{2n-2}\alpha_{n-1}(a;q)}{1-aq^{2n-2}}   \Big), \nonumber \\
\beta_n'(a/q;q)&=\beta_n(a;q).
\end{align}
\end{lemma}

\section{Dual to the lift of Zagier's Example 1}\label{sec-exam1}
From now on we will present some new modular triples discovered by applying the lift-dual operation to Zagier's rank two examples. For convenience, we write $(n_1,n_2,n_3)$ as $(i,j,k)$. The same symbols such as $A,B,C,F(u,v,w;q)$ may have different meanings in different places, but this will not cause any confusion.

Zagier's Example 1 states that for $a>\frac{1}{2}$ ($a\neq 1$) and $b\in \mathbb{Q}$, we have modular triples $(A,B_i,C_i)$ ($i=1,2,3,4$) where
\begin{equation}\label{eq-exam1-original}
    \begin{split}
&A=\begin{pmatrix}
a & 1-a \\
1-a & a
\end{pmatrix},  B_1=\begin{pmatrix} b \\ -b \end{pmatrix}, B_2=\begin{pmatrix}
    -1/2 \\ -1/2
\end{pmatrix}, B_3=\begin{pmatrix}
    1-a/2 \\ a/2
\end{pmatrix},  \\
&B_4=\begin{pmatrix}
    a/2 \\ 1-a/2
\end{pmatrix}, C_1=\frac{b^2}{2a}-\frac{1}{24}, C_2=\frac{1}{8a}-\frac{1}{24}, C_3=\frac{a}{8}-\frac{1}{24}, C_4=\frac{a}{8}-\frac{1}{24}.
\end{split}
\end{equation}
Here the last three vectors were found by Vlasenko--Zwegers \cite{VZ}. The identities justifying their modularity were given in \cite[Table 2]{VZ}:
\begin{align}
\sum_{i,j\geq 0}\frac{q^{\frac{1}{2}ai^2+(1-a)ij+\frac{1}{2}aj^2+b(i-j)}}{(q;q)_i(q;q)_j}&=\frac{1}{(q;q)_\infty} \sum_{n=-\infty}^\infty q^{\frac{1}{2}an^2+2bn}, \label{VZ-id-1}\\
\sum_{i,j\geq 0}\frac{q^{\frac{1}{2}ai^2+(1-a)ij+\frac{1}{2}aj^2-\frac{1}{2}i-\frac{1}{2}j}}{(q;q)_i(q;q)_j}&=\frac{2}{(q;q)_\infty} \sum_{n=-\infty}^\infty q^{\frac{1}{2}an^2+\frac{1}{2}n}, \label{VZ-id-2} \\
\sum_{i,j\geq 0}\frac{q^{\frac{1}{2}ai^2+(1-a)ij+\frac{1}{2}aj^2+(1-\frac{1}{2}a)i+\frac{1}{a}aj}}{(q;q)_i(q;q)_j}&=\frac{1}{2(q;q)_\infty} \sum_{n=-\infty}^\infty q^{\frac{1}{2}an^2+\frac{1}{2}an}. \label{VZ-id-3}
\end{align}

This lifts to rank three modular triples $(\widetilde{A},\widetilde{B}_i,C_i)$ where $C_i$ is as above and
\begin{equation}\label{exam1-lift}
\begin{split}
&\widetilde{A}=\begin{pmatrix} a & 2-a & 1 \\  2-a & a & 1 \\ 1 & 1 & 2 \end{pmatrix}, ~~ \widetilde{B}_1=\begin{pmatrix} b \\ -b \\ 0 \end{pmatrix}, ~~ \widetilde{B}_2= \begin{pmatrix} -1/2 \\ -1/2 \\ -1 \end{pmatrix}, \\
&\widetilde{B}_3=\begin{pmatrix} 1-a/2 \\ a/2 \\ 1 \end{pmatrix}, ~~ \widetilde{B}_4=\begin{pmatrix} a/2 \\ 1-a/2 \\ 1 \end{pmatrix}.
\end{split}
\end{equation}
The matrix $\widetilde{A}$ also appeared in a recent work of Gang--Kim--Park--Stubbs \cite[Sec.\ 4.2.3]{GKPS}. Considering its dual, we obtain possible modular triples $(\widetilde{A}^\star,\widetilde{B}_i^\star,C_i^\star)$ where
\begin{equation}\label{A-exam1-lift-dual}
\begin{split}
&\widetilde{A}^\star=\begin{pmatrix} \frac{2a-1}{4(a-1)} & \frac{2a-3}{4(a-1)} & -\frac{1}{2} \\ \frac{2a-3}{4(a-1)} & \frac{2a-1}{4(a-1)} & -\frac{1}{2} \\ -\frac{1}{2} & -\frac{1}{2} & 1 \end{pmatrix},  ~~
\widetilde{B}_1^\star=\begin{pmatrix} b/2(a-1) \\ -b/2(a-1) \\ 0 \end{pmatrix}, ~~ \widetilde{B}_2^\star=\begin{pmatrix} 0 \\ 0 \\ -1/2 \end{pmatrix}, \\
& \widetilde{B}_3^\star=\begin{pmatrix} -1/4 \\ 1/4  \\ 1/2 \end{pmatrix}, ~~ \widetilde{B}_4^\star=\begin{pmatrix} 1/4 \\ -1/4  \\ 1/2 \end{pmatrix}, ~~ C_1^\star=\frac{b^2}{2a(a-1)}-\frac{1}{12}, \\
& C_2^\star=\frac{1}{6}-\frac{1}{8a}, ~~ C_3^\star=\frac{1}{24},  ~~ C_4^\star=\frac{1}{24}.
\end{split}
\end{equation}
We may rewrite them as
\begin{align}\label{eq-vector-B-dual-exam1}
   & \widetilde{A}^\star=\begin{pmatrix} \frac{1}{2}+m & \frac{1}{2}-m & -\frac{1}{2} \\ \frac{1}{2}-m & \frac{1}{2}+m & -\frac{1}{2} \\ -\frac{1}{2} & -\frac{1}{2} & 1 \end{pmatrix}, ~~\widetilde{B}_1^\star=\begin{pmatrix} \nu \\ -\nu \\ 0 \end{pmatrix},\widetilde{B}_2^\star=\begin{pmatrix} 0 \\ 0 \\ -1/2 \end{pmatrix},  \widetilde{B}_3^\star=\begin{pmatrix} -1/4 \\ 1/4  \\ 1/2 \end{pmatrix}, \nonumber \\
    &\widetilde{B}_4^\star=\begin{pmatrix} 1/4 \\ -1/4  \\ 1/2 \end{pmatrix},~~ C_1^\star=\frac{2\nu^2}{(4m+1)}-\frac{1}{12}, ~~ C_2^\star=\frac{m+1}{6(4m+1)}, ~~ C_3^\star=\frac{1}{24},  ~~ C_4^\star=\frac{1}{24}.
\end{align}

We now prove that they are indeed modular. Since $i$ and $j$ are symmetric in the quadratic form generated by $\widetilde{A}^\star$, there are essentially only three Nahm sums to consider.
\begin{theorem}\label{2-Ex1-in}
We have
\begin{align}
&\sum_{i,j,k\geq 0} \frac{q^{2m(i-j)^2+(i+j)^2+2k^2-2(i+j)k+\nu(i-j)}}{(q^4;q^4)_i(q^4;q^4)_j(q^4;q^4)_k} \nonumber \\
&=\frac{J_4^3\overline{J}_{2(4m+\nu+1),4(4m+1)}}{J_2^2J_8^2} +2q^{2m+\nu+1} \frac{J_8^2\overline{J}_{-2\nu,4(4m+1)}}{J_4^3}, \label{thm1-id-1}\\
&\sum_{i,j,k\geq 0} \frac{q^{2m(i-j)^2+(i+j)^2+2k^2-2(i+j)k-2k}}{(q^4;q^4)_i(q^4;q^4)_j(q^4;q^4)_k} =4\frac{J_{8}^{2}\overline{J}_{8m,16m+4}}{J_{4}^{3}}+2q^{2m-1}\frac{J_{4}^{3}\overline{J}_{2,16m+4}}{J_{2}^{2}J_{8}^{2}}, \label{thm1-id-2}\\
&\sum_{i,j,k\geq 0} \frac{q^{2m(i-j)^2+(i+j)^2+2k^2-2(i+j)k+i-j+2k}}{(q^4;q^4)_i(q^4;q^4)_j(q^4;q^4)_k}=\frac{J_8^2\overline{J}_{8m+2,16m+4}}{J_4^3}+q^{2m}\frac{J_4^3J_{32m+8}^2}{J_2^2J_8^2J_{16m+4}}. \label{thm1-id-3}
\end{align}
\end{theorem}
\begin{proof}
We define
\begin{align}
F(u,v,w;q^4)=\sum_{i,j,k\geq 0} \frac{u^iv^jw^kq^{2m(i-j)^2+(i+j)^2+2k^2-2(i+j)k}}{(q^4;q^4)_i(q^4;q^4)_j(q^4;q^4)_k}.
\end{align}
We have
\begin{align}
&F(u,u^{-1},w;q^4)=\sum_{i,j,k\geq 0} \frac{q^{2m(i-j)^2+(i+j)^2+2k^2-2(i+j)k}u^{i-j}w^k}{(q^4;q^4)_i(q^4;q^4)_j(q^4;q^4)_k}\nonumber\\
&=\sum_{i,j\geq 0} \frac{q^{2m(i-j)^2+(i+j)^2}u^{i-j}(-wq^{2-2(i+j)};q^4)_{\infty}}{(q^4;q^4)_i(q^4;q^4)_j}\quad \text{(set $i+j=n$}) \nonumber\\
&=\sum_{n=0}^{\infty}\sum_{j=0}^{n} \frac{q^{2m(n-2j)^2+n^2}u^{n-2j}(-wq^{2-2n};q^4)_{\infty}}{(q^4;q^4)_j(q^4;q^4)_{n-j}} \nonumber\\
&=S_0(q)+S_1(q) \label{1-proof-S},
\end{align}
where $S_0(q)$ and $S_1(q)$ correspond to the sum with $n$ even and odd, respectively, namely,
\begin{align}
S_0(q)&=(-wq^{2};q^4)_{\infty}\sum_{n=0}^\infty \sum_{j=0}^{2n} \frac{q^{8m(n-j)^2+2n^2}u^{2(n-j)}w^n(-q^{2}/w;q^4)_{n}}{(q^4;q^4)_j(q^4;q^4)_{2n-j}}, \\
S_1(q)&=(-w;q^4)_{\infty}\sum_{n=0}^\infty \sum_{j=0}^{2n+1} \frac{q^{8m(n-j)^2+8m(n-j)+2m+2n^2+2n+1}u^{2(n-j)+1}w^n(-q^{4}/w;q^4)_{n}}{(q^4;q^4)_j(q^4;q^4)_{2n-j+1}}.
\end{align}
We have
\begin{align}
&S_0(q)=(-wq^{2};q^4)_{\infty}\sum_{n=0}^\infty w^{n}q^{2n^2}(-q^{2}/w;q^4)_{n}
\Big(\sum_{r=0}^{n} \frac{q^{8mr^2}u^{2r}}{(q^4;q^4)_{n-r}(q^4;q^4)_{n+r}} \nonumber \\
&\qquad \qquad +\sum_{r=1}^{n} \frac{q^{8mr^2}u^{-2r}}{(q^4;q^4)_{n-r}(q^4;q^4)_{n+r}}\Big) \nonumber\\
&=(-wq^{2};q^4)_{\infty}\sum_{n=0}^\infty w^nq^{2n^2}(-q^{2}/w;q^4)_{n} \beta_n(1;q^4).
\label{1-proof-S0}
\end{align}
Here  $\beta_n(1;q^4)$ are defined by \eqref{defn-BP} with
\begin{align*}
    \alpha_r(1;q^4):=\left\{\begin{array}{ll}
    1 & r=0, \\
    q^{8mr^2}(u^{2r}+u^{-2r}) & r\geq 1.
    \end{array} \right.
\end{align*}
By \eqref{Bailey's lemma-1} we obtain
\begin{align}
S_0(q)=\frac{(-wq^{2};q^4)_{\infty}^2}{(q^4;q^4)_\infty}\Big(1+\sum_{r=1}^{\infty}\frac{(-q^2/w;q^4)_{r}}{(-q^2w;q^4)_{r}}w^rq^{(8m+2)r^2}(u^{2r}+u^{-2r})\Big).  \label{id-S0}
\end{align}

Similarly, we have
\begin{align}
&S_1(q)=(-w;q^4)_{\infty}\sum_{n=0}^{\infty}w^nq^{2n^2+2n+1}(-q^{4}/w;q^4)_{n}
\Big(\sum_{r=0}^{n} \frac{q^{8mr^2+8mr+2m}u^{2r+1}}{(q^4;q^4)_{n-r}(q^4;q^4)_{n+r+1}}\nonumber\\
&\qquad  \qquad +\sum_{r=0}^{n} \frac{q^{8mr^2+8mr+2m}u^{-2r-1}}{(q^4;q^4)_{n-r}(q^4;q^4)_{n+r+1}}\Big) \nonumber\\
&= \frac{q^{2m+1}}{1-q^4}(-w;q^4)_{\infty}\sum_{n=0}^{\infty}w^nq^{2n^2+2n}(-q^{4}/w;q^4)_{n}\beta_n(q^4;q^4). \label{1-proof-S1}
\end{align}
Here  $\beta_n(q^4;q^4)$ are defined by \eqref{defn-BP} with
\begin{align*}
\alpha_r(q^4;q^4)=q^{8mr^2+8mr}(u^{2r+1}+u^{-2r-1}).
\end{align*}By \eqref{Bailey's lemma-2} we obtain
\begin{align}
S_1(q)&=\frac{q^{2m+1}(-w,-wq^{4};q^4)_{\infty}}{(q^4;q^4)_\infty} \nonumber \\
&\qquad \times \sum_{r=0}^{\infty}\frac{(-q^4/w;q^4)_{r}}{(-q^4w;q^4)_{r}}w^rq^{(8m+2)r^2+(8m+2)r}(u^{2r+1}+u^{-2r-1}).  \label{id-S1}
\end{align}
Substituting \eqref{id-S0} and \eqref{id-S1} into \eqref{1-proof-S}, we deduce that
\begin{align}\label{exam1-S-result}
    &F(u,u^{-1},w;q^4)=\frac{(-wq^{2};q^4)_{\infty}^2}{(q^4;q^4)_\infty}\Big(1+\sum_{r=1}^{\infty}\frac{(-q^2/w;q^4)_{r}}{(-q^2w;q^4)_{r}}w^rq^{(8m+2)r^2}(u^{2r}+u^{-2r})\Big) \nonumber \\
    & +\frac{q^{2m+1}(-w,-wq^{4};q^4)_{\infty}}{(q^4;q^4)_\infty}\sum_{r=0}^{\infty}\frac{(-q^4/w;q^4)_{r}}{(-q^4w;q^4)_{r}}w^rq^{(8m+2)r^2+(8m+2)r}(u^{2r+1}+u^{-2r-1}).
\end{align}

(1) Setting $w=1$ in \eqref{exam1-S-result} and using \eqref{JTP}, we have
\begin{align}
&F(u,u^{-1},1;q^4)=\frac{(-q^{2};q^4)_{\infty}^2}{(q^4;q^4)_\infty}\sum_{r=-\infty}^{\infty}q^{(8m+2)r^2}u^{2r}\nonumber\\
&\qquad +\frac{q^{2m+1}(-1,-q^{4};q^4)_{\infty}}{(q^4;q^4)_\infty}\sum_{r=-\infty}^{\infty}q^{(8m+2)r^2+(8m+2)r}u^{2r+1} \nonumber\\
&=\frac{J_{4}^{3}(-u^2q^{8m+2},-q^{8m+2}/u^2,q^{16m+4};q^{16m+4})_{\infty}}{J_{2}^{2}J_{8}^{2}}\nonumber\\
&\qquad +\frac{2uq^{2m+1}J_{8}^{2}(-u^2q^{16m+4},-1/u^2,q^{16m+4};q^{16m+4})_{\infty}}{J_{4}^{3}}.
\end{align}
In particular, when $u=q^\nu$ we obtain \eqref{thm1-id-1}.

(2) Setting $w=q^{-2}$ in \eqref{exam1-S-result} and using \eqref{JTP}, we have
\begin{align}
&F(u,u^{-1},q^{-2};q^4)=\frac{(-1;q^4)_{\infty}^2}{(q^4;q^4)_\infty}\Big(1+\sum_{r=1}^{\infty}\frac{1+q^{4r}}{2}q^{(8m+2)r^2-2r}(u^{2r}+u^{-2r})\Big)\nonumber\\
&\qquad +\frac{q^{2m+1}(-q^{-2},-q^{2};q^4)_{\infty}}{(q^4;q^4)_\infty}\sum_{r=0}^{\infty}\frac{1+q^{4r+2}}{1+q^2}q^{(8m+2)r^2+8mr}(u^{2r+1}+u^{-2r-1}) \nonumber\\
&=\frac{2J_{8}^{2}}{J_{4}^{3}}\sum_{r=-\infty}^{\infty}q^{(8m+2)r^2-2r}(u^{2r}+u^{-2r}) \nonumber \\
&\qquad \qquad +q^{2m-1}\frac{J_{4}^{3}}{J_{2}^{2}J_{8}^{2}}\sum_{r=-\infty}^{\infty}q^{(8m+2)r^2+8mr}(u^{2r+1}+u^{-2r-1}) \nonumber\\
&=2\frac{J_{8}^{2}}{J_{4}^{3}}\Big((-u^2q^{8m},-q^{8m+4}/u^2,q^{16m+4};q^{16m+4})_{\infty} \nonumber\\
&\qquad \qquad +(-u^2q^{8m+4},-q^{8m}/u^2,q^{16m+4};q^{16m+4})_{\infty} \Big) \nonumber \\
&\qquad +q^{2m-1}\frac{J_{4}^{3}}{J_{2}^{2}J_{8}^{2}}\Big(u(-u^2q^{16m+2},-q^{2}/u^2,q^{16m+4};q^{16m+4})_{\infty} \nonumber \\
&\qquad  \qquad +u^{-1}(-u^2q^{2},-q^{16m+2}/u^2,q^{16m+4};q^{16m+4})_{\infty} \Big).
\end{align}
 In particular, when $u=1$ we obtain \eqref{thm1-id-2}.

(3) Setting $w=q^2$ in \eqref{exam1-S-result}, we have
\begin{align}
&F(u,u^{-1},q^{2};q^4)=\frac{(-q^4;q^4)_{\infty}^2}{(q^4;q^4)_\infty}\Big(1+\sum_{r=1}^{\infty}\frac{2}{1+q^{4r}}q^{(8m+2)r^2+2r}(u^{2r}+u^{-2r})\Big)\\
&\qquad +\frac{q^{2m+1}(-q^{2},-q^{6};q^4)_{\infty}}{(q^4;q^4)_\infty}\sum_{r=0}^{\infty}\frac{1+q^2}{1+q^{4r+2}}q^{(8m+2)r^2+(8m+4)r}(u^{2r+1}+u^{-2r-1}).     \nonumber
\end{align}
Setting $u=q$ and using \eqref{JTP}, we deduce that
\begin{align*}
&F(q,q^{-1},q^{2};q^4)=\frac{(-q^4;q^4)_{\infty}^2}{(q^4;q^4)_\infty}\Big(1+2\sum_{r=1}^{\infty}q^{(8m+2)r^2}\Big)\nonumber\\
&\qquad +\frac{q^{2m}(-q^{2};q^4)_{\infty}^{2}}{(q^4;q^4)_\infty}\sum_{r=0}^{\infty}q^{(8m+2)r^2+(8m+2)r}\nonumber\\
&=\frac{(-q^4;q^4)_{\infty}^2(-q^{8m+2},-q^{8m+2},q^{16m+4};q^{16m+4})_{\infty}}{(q^4;q^4)_\infty}\nonumber\\
&\qquad +\frac{q^{2m}(-q^{2};q^4)_{\infty}^{2}(-1,-q^{16m+4},q^{16m+4};q^{16m+4})_{\infty}}{2(q^4;q^4)_\infty}. \qedhere
\end{align*}
\end{proof}

\section{Dual to the lift of Zagier's Example 9}\label{sec-exam9}
Example 9 provides three modular triples $(A,B_i,C_i)$ ($i=1,2,3$) where
\begin{align}
    &A=\begin{pmatrix}
        1 & -1/2 \\ -1/2 & 3/4
    \end{pmatrix}, \quad B_1=\begin{pmatrix}
        -1/2 \\ 1/4
    \end{pmatrix}, ~~ B_2=\begin{pmatrix}
        0 \\ 0
    \end{pmatrix}, ~~ B_3=\begin{pmatrix}
        0 \\ 1/2
    \end{pmatrix}\nonumber \\
    &C_1=1/28, \quad C_2=-3/56, \quad C_3=1/56.
\end{align}
This lifts to the modular triples $(\widetilde{A},\widetilde{B}_i,C_i)$ where
\begin{align}
\widetilde{A}=\begin{pmatrix} 1 & 1/2 & 1/2 \\  1/2 & 3/4 & 1/4 \\ 1/2 & 1/4 & 3/4 \end{pmatrix}, \quad \widetilde{B}_1=
\begin{pmatrix} -1/2 \\ 1/4 \\ -1/4 \end{pmatrix}, \widetilde{B}_2=  \begin{pmatrix} 0 \\ 0 \\ 0 \end{pmatrix} , \widetilde{B}_3=\begin{pmatrix} 0 \\ 1/2 \\ 1/2 \end{pmatrix}.
\end{align}
Considering its dual, we expect that $(\widetilde{A}^\star,\widetilde{B}_i^\star,C_i^\star)$ are modular triples where
\begin{align}
&\widetilde{A}^\star=\begin{pmatrix} 2 & -1 & -1 \\  -1 & 2 & 0 \\ -1 & 0 & 2 \end{pmatrix}, \quad \widetilde{B}_1^\star=
\begin{pmatrix} -1 \\ 1 \\ 0 \end{pmatrix},  \quad \widetilde{B}_2^\star=\begin{pmatrix} 0 \\ 0 \\ 0 \end{pmatrix}, \quad \widetilde{B}_3^\star=\begin{pmatrix} -1 \\ 1 \\ 1 \end{pmatrix} \nonumber  \\
&C_1^\star=3/14, \quad C_2^\star=-1/14, \quad C_3^\star=5/14.
\end{align}
The matrix $\widetilde{A}^\star$ is essentially the Cartan matrix of the Dynkin diagram $A_3$ (see also \cite[Sec.\ 4.1.3]{GKPS}). We establish the following identities to prove their modularity.
\begin{theorem}\label{thm-3}
We have
\begin{align}
&\sum_{i,j,k\geq 0} \frac{q^{2i^2+2j^2+2k^2-2ij-2ik-2i+2j}}{(q^2;q^2)_i(q^2;q^2)_j(q^2;q^2)_k} \label{Thm2.4-1} \\
&=4 \frac{J_4^{6}J_{28}^{3}}{J_2^{6}J_{4,28}J_{6,28}J_{8,28}}-\frac{1}{2}\frac{J_{1}^{5}J_{7}J_{1,14}J_{3,14}}{J_{2}^{5}J_{2,14}J_{6,14}J_{14}}-\frac{1}{2}\frac{J_{2}^{11}J_{5,14}J_{4,28}}{J_1^{6}J_{4}^{6}J_{28}}, \nonumber \\
&\sum_{i,j,k\geq 0} \frac{q^{2i^2+2j^2+2k^2-2ij-2ik}}{(q^2;q^2)_i(q^2;q^2)_j(q^2;q^2)_k}  \label{Thm2.4-2} \\
&=\frac{1}{2}\frac{J_{2}^{11}J_{1,14}J_{12,28}}{J_1^{6}J_{4}^{6}J_{28}}
+\frac{1}{2}\frac{J_{1}^{5}J_{7}J_{3,14}J_{5,14}}{J_{2}^{5}J_{4,14}J_{6,14}J_{14}}
-4q^2 \frac{J_4^{6}J_{28}^{3}}{J_2^{6}J_{4,28}J_{10,28}J_{12,28}}, \nonumber \\&\sum_{i,j,k\geq 0} \frac{q^{2i^2+2j^2+2k^2-2ij-2ik-2i+2j+2k}}{(q^2;q^2)_i(q^2;q^2)_j(q^2;q^2)_k} \label{Thm2.4-3}
\\
&=\frac{1}{2q}\frac{J_{2}^{11}J_{14}J_{3,28}J_{11,28}}{J_1^{6}J_{4}^{5}J_{28}J_{4,28}J_{12,28}}
-\frac{1}{2q}\frac{J_{1}^{5}J_{7}J_{1,14}J_{5,14}}{J_{2}^{5}J_{2,14}J_{4,14}J_{14}}
-4\frac{J_4^{6}J_{28}^{3}}{J_2^{6}J_{2,28}J_{8,28}J_{12,28}}.  \nonumber
\end{align}
\end{theorem}
\begin{proof}
We define
\begin{align}
F(u,v,w;q^2):=\sum_{i,j,k\geq 0} \frac{u^iv^jw^kq^{2i^2+2j^2+2k^2-2ij-2ik}}{(q^2;q^2)_i(q^2;q^2)_j(q^2;q^2)_k}.
\end{align}
By \eqref{Euler1} and \eqref{JTP} we have
\begin{align}\label{integration1}
&F(u,v,w;q^2)=
\sum_{i,j,k\geq 0} \frac{u^iv^jw^kq^{(i-j-k)^2+(j-k)^2+i^2}}{(q^2;q^2)_i(q^2;q^2)_j(q^2;q^2)_k} \nonumber\\
&=\oint \oint \sum_{i\geq 0}\frac{(uz)^iq^{i^2}}{(q^2;q^2)_i}\sum_{j\geq 0}\frac{(vy/z)^j}{(q^2;q^2)_j}\sum_{k\geq 0}\frac{(w/yz)^k}{(q^2;q^2)_k}\sum_{m=-\infty}^{\infty}z^{-m}q^{m^2}\sum_{n=-\infty}^{\infty}y^{-n}q^{n^2}\frac{dy}{2\pi iy}\frac{dz}{2\pi iz} \nonumber\\
&=\oint \oint \frac{(-quz,-qz,-q/z,q^2;q^2)_{\infty}(-qy,-q/y,q^2;q^2)_{\infty}}{(vy/z,w/yz;q^2)_{\infty}}\frac{dy}{2\pi iy}\frac{dz}{2\pi iz}\nonumber\\
&=\oint (-quz,-qz,-q/z,q^2,q^2;q^2)_{\infty} \nonumber\\
&\quad \times \oint \sum_{i\geq 0}\frac{(-qz/v;q^2)_i}{(q^2;q^2)_i}(vy/z)^i\sum_{j\geq 0}\frac{(-qz/w;q^2)_j}{(q^2;q^2)_j}(w/yz)^j \frac{dy}{2\pi iy}\frac{dz}{2\pi iz}\quad \text{(by (\ref{q-binomial}))} \nonumber\\
&=\oint (-quz,-qz,-q/z,q^2,q^2;q^2)_{\infty}\sum_{i\geq 0}\frac{(-qz/v,-qz/w;q^2)_i}{(q^2,q^2;q^2)_i}(\frac{vw}{z^2})^i \frac{dz}{2\pi iz} \nonumber\\
&=\oint (-quz,-qz,-q/z,q^2,q^2;q^2)_{\infty}\frac{(-qv/z,-qw/z;q^2)_{\infty}}{(vw/z^2,q^2;q^2)_{\infty}} \frac{dz}{2\pi iz} \quad \text{(by (\ref{Gauss}))} \nonumber\\
&=\oint \frac{(-quz,-qv/z,-qw/z,-qz,-q/z,q^2;q^2)_{\infty}}{(vw/z^2;q^2)_{\infty}} \frac{dz}{2\pi iz}.
\end{align}
(1) By \eqref{integration1}, the left side of \eqref{Thm2.4-1} is the same as
\begin{align}
F(q^{-2},q^2,1;q^2)&=\oint \frac{(-z/q,-q^3/z,-q/z,-qz,-q/z,q^2;q^2)_{\infty}}{(q^2/z^2;q^2)_{\infty}} \frac{dz}{2\pi iz} \nonumber\\
&=\oint \frac{(-z/q,-qz,-q/z,-q^3/z,q^2;q^2)_{\infty}}{(q/z,-q^2/z,q^2/z;q^2)_{\infty}} \frac{dz}{2\pi iz}.
\end{align}
Applying \eqref{Eq. (4.10.5)} with
$$(A,B,C,D)=(2,2,0,3),$$
$$(a_1,a_2)=(-1/q,-q),\quad (b_1,b_2)=(-q,-q^3),\quad (d_1,d_2,d_3)=(q,-q^2,q^2),$$
we deduce that
\begin{align}\label{Thm2.4-1-R}
F(q^{-2},q^2,1;q^2)=J_2(R_1(q)+R_2(q)+R_3(q)),
\end{align}
where
\begin{align}
&R_1(q)=\frac{(-1,-q^2,-1,-q^2;q^2)_{\infty}}{(q^2,-q,q;q^2)_{\infty}}\sum_{n\geq 0}\frac{(-q^2,-1;q^2)_{n}}{(q^2,-1,-q^2,-q,q;q^2)_{n}}q^{n^2+n},\\
&R_2(q)=\frac{(q,q^3,1/q,q;q^2)_{\infty}}{(q^2,-1/q,-1;q^2)_{\infty}}\sum_{n\geq 0}\frac{(q^3,q;q^2)_{n}}{(q^2,q,q^3,-q^3,-q^2;q^2)_{n}}(-1)^nq^{n^2+2n},\\
&R_3(q)=\frac{(-q,-q^3,-1/q,-q;q^2)_{\infty}}{(q^2,1/q,-1;q^2)_{\infty}}\sum_{n\geq 0}\frac{(-q^3,-q;q^2)_{n}}{(q^2,-q,-q^3,q^3,-q^2;q^2)_{n}}q^{n^2+2n}.
\end{align}
By \eqref{S. 81} with $q$ replaced by $q^2$, we have
\begin{align}\label{Thm2.4-1-R1}
R_1(q)=4\frac{(q^4;q^4)_{\infty}^{4}}{(q^2;q^2)_{\infty}^{5}(q^2;q^4)_{\infty}}\sum_{n\geq 0}\frac{q^{n^2+n}}{(q^2;q^2)_{n}(q^2;q^4)_{n}}
=4\frac{J_{4}^{6}J_{28}^{3}}{J_{2}^{7}J_{4,28}J_{6,28}J_{8,28}}.
\end{align}
Similarly, by \eqref{Entry 3.5.6} we have
\begin{align}\label{Thm2.4-1-R2}
R_2(q)=-\frac{1}{2}\frac{(q;q^2)_{\infty}^{4}}{(q^2;q^2)_{\infty}(-q;q)_{\infty}}\sum_{n\geq 0}\frac{(-1)^nq^{n^2+2n}}{(q^4;q^4)_{n}(-q;q^2)_{n+1}}
=-\frac{1}{2}\frac{J_{1}^{5}J_{1,14}J_{3,14}J_{7}}{J_{2}^{6}J_{2,14}J_{6,14}J_{14}}.
\end{align}
Next, by \eqref{S. 119} we have
\begin{align}\label{Thm2.4-1-R3}
R_3(q)=-\frac{1}{2}\frac{(-q;q^2)_{\infty}^{4}}{(q;q)_{\infty}(-q^2;q^2)_{\infty}}\sum_{n\geq 0}\frac{q^{n^2+2n}}{(q^4;q^4)_{n}(q;q^2)_{n+1}}
=-\frac{1}{2}\frac{J_{2}^{10}J_{5,14}J_{4,28}}{J_{1}^{6}J_{4}^{6}J_{28}}.
\end{align}
Now substituting \eqref{Thm2.4-1-R1}--\eqref{Thm2.4-1-R3} into \eqref{Thm2.4-1-R}, we get \eqref{Thm2.4-1}.

(2) By (\ref{integration1}), the left side of (\ref{Thm2.4-2}) is the same as
\begin{align}
F(1,1,1;q^2)&=\oint \frac{(-qz,-q/z,-q/z,-qz,-q/z,q^2;q^2)_{\infty}}{(1/z^2;q^2)_{\infty}} \frac{dz}{2\pi iz} \nonumber\\
&=\oint \frac{(-qz,-qz,-q/z,-q/z,q^2;q^2)_{\infty}}{(1/z,-1/z,q/z;q^2)_{\infty}} \frac{dz}{2\pi iz}.
\end{align}
Applying (\ref{Eq. (4.10.5)}) with
$$(A,B,C,D)=(2,2,0,3),$$
$$(a_1,a_2)=(-q,-q),\quad (b_1,b_2)=(-q,-q),\quad (d_1,d_2,d_3)=(1,-1,q),$$
we deduce that
\begin{align}\label{Thm2.4-2-S}
F(1,1,1;q^2)=J_2(S_1(q)+S_2(q)+S_3(q)),
\end{align}
where
\begin{align}
&S_1(q)=\frac{(-q,-q,-q,-q;q^2)_{\infty}}{(q^2,-1,q;q^2)_{\infty}}\sum_{n\geq 0}\frac{(-q,-q;q^2)_{n}}{(q^2,-q,-q,-q^2,q;q^2)_{n}}q^{n^2+2n},\\
&S_2(q)=\frac{(q,q,q,q;q^2)_{\infty}}{(q^2,-1,-q;q^2)_{\infty}}\sum_{n\geq 0}\frac{(q,q;q^2)_{n}}{(q^2,q,q,-q^2,-q;q^2)_{n}}(-1)^nq^{n^2+2n},\\
&S_3(q)=\frac{(-q^2,-q^2,-1,-1;q^2)_{\infty}}{(q^2,1/q,-1/q;q^2)_{\infty}}\sum_{n\geq 0}\frac{(-q^2,-q^2;q^2)_{n}}{(q^2,-q^2,-q^2,q^3,-q^3;q^2)_{n}}q^{n^2+3n}.
\end{align}
By \eqref{S. 118} we have
\begin{align}\label{Thm2.4-2-S1}
S_1(q)=\frac{1}{2}\frac{(q^2;q^4)_{\infty}^{4}}{(q;q^2)_{\infty}^{5}(q^4;q^4)_{\infty}}\sum_{n\geq 0}\frac{q^{n^2+2n}}{(q;q^2)_{n}(q^4;q^4)_{n}}
=\frac{1}{2}\frac{J_{2}^{10}J_{1,14}J_{12,28}}{J_{1}^{6}J_{4}^{6}J_{28}}.
\end{align}
Similarly, by \eqref{Entry 3.5.5} we have
\begin{align}\label{Thm2.4-2-S2}
S_2(q)=\frac{1}{2}\frac{(q;q^2)_{\infty}^{5}}{(q^2;q^2)_{\infty}}\sum_{n\geq 0}\frac{(-1)^nq^{n^2+2n}}{(q^4;q^4)_{n}(-q;q^2)_{n}}
=\frac{1}{2}\frac{J_{1}^{5}J_{3,14}J_{5,14}J_{7}}{J_{2}^{6}J_{4,14}J_{6,14}J_{14}}.
\end{align}
Next, by \eqref{S. 82} with $q$ replaced by $q^2$, we have
\begin{align}\label{Thm2.4-2-S3}
S_3(q)=-4q^2\frac{(q^4;q^4)_{\infty}^{4}}{(q^2;q^2)_{\infty}^{5}(q^2;q^4)_{\infty}}\sum_{n\geq 0}\frac{q^{n^2+3n}}{(q^2;q^2)_{n}(q^2;q^4)_{n+1}}
=-4q^2\frac{J_{4}^{6}J_{28}^{3}}{J_{2}^{7}J_{4,28}J_{10,28}J_{12,28}}.
\end{align}
Now substituting \eqref{Thm2.4-2-S1}--\eqref{Thm2.4-2-S3} into \eqref{Thm2.4-2-S}, we get \eqref{Thm2.4-2}.

(3) By (\ref{integration1}), the left side of (\ref{Thm2.4-3}) is the same as
\begin{align}
F(q^{-2},q^2,q^2;q^2)&=\oint \frac{(-z/q,-q^3/z,-q^3/z,-qz,-q/z,q^2;q^2)_{\infty}}{(q^4/z^2;q^2)_{\infty}} \frac{dz}{2\pi iz} \nonumber\\
&=\oint \frac{(-z/q,-qz,-q^3/z,-q/z,q^2;q^2)_{\infty}}{(q^2/z,-q^2/z,q^3/z;q^2)_{\infty}} \frac{dz}{2\pi iz}.
\end{align}
Applying (\ref{Eq. (4.10.5)}) with
$$(A,B,C,D)=(2,2,0,3),$$
$$(a_1,a_2)=(-1/q,-q),\quad (b_1,b_2)=(-q^3,-q),\quad (d_1,d_2,d_3)=(q^2,-q^2,q^3),$$
we deduce that
\begin{align}\label{Thm2.4-3-T}
F(q^{-2},q^2,q^2;q^2)=J_2(T_1(q)+T_2(q)+T_3(q)),
\end{align}
where
\begin{align}
&T_1(q)=\frac{(-q,-q^3,-q,-1/q;q^2)_{\infty}}{(q^2,-1,q;q^2)_{\infty}}\sum_{n\geq 0}\frac{(-q,-q^3;q^2)_{n}}{(q^2,-q,-q^3,-q^2,q;q^2)_{n}}q^{n^2},\\
&T_2(q)=\frac{(q,q^3,q,1/q;q^2)_{\infty}}{(q^2,-1,-q;q^2)_{\infty}}\sum_{n\geq 0}\frac{(q,q^3;q^2)_{n}}{(q^2,q,q^3,-q^2,-q;q^2)_{n}}(-1)^nq^{n^2},\\
&T_3(q)=\frac{(-q^2,-q^4,-1,-1/q^2;q^2)_{\infty}}{(q^2,1/q,-1/q;q^2)_{\infty}}\sum_{n\geq 0}\frac{(-q^2,-q^4;q^2)_{n}}{(q^2,-q^2,-q^4,q^3,-q^3;q^2)_{n}}q^{n^2+n}.
\end{align}
By \eqref{S. 117} we have
\begin{align}\label{Thm2.4-3-T1}
T_1(q)=\frac{1}{2q}\frac{(-q;q^2)_{\infty}^{4}}{(q;q^2)_{\infty}(q^4;q^4)_{\infty}}\sum_{n\geq 0}\frac{q^{n^2}}{(q;q^2)_{n}(q^4;q^4)_{n}}
=\frac{1}{2q}\frac{J_{2}^{10}J_{14}J_{3,28}J_{11,28}}{J_{1}^{6}J_{4}^{5}J_{28}J_{4,28}J_{12,28}}.
\end{align}
Similarly, by \eqref{Entry 3.5.4} we have
\begin{align}\label{Thm2.4-3-T2}
T_2(q)=-\frac{1}{2q}\frac{(q;q^2)_{\infty}^{4}}{(-q;q^2)_{\infty}(q^4;q^4)_{\infty}}\sum_{n\geq 0}\frac{(-1)^nq^{n^2}}{(q^4;q^4)_{n}(-q;q^2)_{n}}
=-\frac{1}{2q}\frac{J_{1}^{5}J_{1,14}J_{5,14}J_{7}}{J_{2}^{6}J_{2,14}J_{4,14}J_{14}}.
\end{align}
Next, by \eqref{S. 80} with $q$ replaced by $q^2$, we have
\begin{align}\label{Thm2.4-3-T3}
T_3(q)=-4\frac{(-q^2;q^2)_{\infty}^{4}}{(q^2;q^2)_{\infty}(q^2;q^4)_{\infty}}\sum_{n\geq 0}\frac{q^{n^2+n}}{(q^2;q^2)_{n}(q^2;q^4)_{n+1}}
=-4\frac{J_{4}^{6}J_{28}^{3}}{J_{2}^{7}J_{2,28}J_{8,28}J_{12,28}}.
\end{align}
Now substituting \eqref{Thm2.4-3-T1}--\eqref{Thm2.4-3-T3} into \eqref{Thm2.4-3-T}, we get \eqref{Thm2.4-3}.
\end{proof}
\begin{rem}
The idea behind the deduction of \eqref{integration1} comes from Wang's proof \cite[Theorem 4.8]{Wang-rank3} of Zagier's seventh example on rank three Nahm sums \cite[Table 3]{Zagier}. This will also be used in \eqref{integration2} in Section \ref{sec-exam11}.
\end{rem}

\section{Dual to the lift of Zagier's Example 11}\label{sec-exam11}
Surprisingly, the proof of Theorem \ref{thm-lift-11} will rely on the following identities associated with Zagier's Example 10:
\begin{align}
\sum_{i,j\geq 0} \frac{q^{2i^2+2ij+2j^2}}{(q^3;q^3)_i(q^3;q^3)_j}
&=\frac{1}{J_3}\left(J_{21,45}-q^3J_{6,45}+2q^2J_{9,45}  \right), \label{conj-10-2} \\
\sum_{i,j\geq 0} \frac{q^{2i^2+2ij+2j^2-2i-j}}{(q^3;q^3)_i(q^3;q^3)_j}&=\frac{1}{J_3}\left(2J_{18,45}+qJ_{12,45}+q^4J_{3,45}\right). \label{conj-10-1}
\end{align}
They were conjectured by Vlasenko--Zwegers \cite[p.\ 633, Table 1]{VZ} and proved by Cao--Rosengren--Wang \cite{CRW} through purely $q$-series approach.

The identities \eqref{conj-10-2} and  \eqref{conj-10-1} give 3-dissection formulas for the Nahm sums on the left side. We record the following consequence which play a key role in our proof of Theorem \ref{thm-lift-11}.
\begin{lemma}\label{lem-3-dissection}
For $r\in \{-1,0,1\}$ we define
\begin{align}
    S_r(q):=\sum_{\begin{smallmatrix}
        i,j\geq 0 \\ i-j\equiv r \!\!\! \pmod{3}
    \end{smallmatrix}} \frac{q^{2i^2+2ij+2j^2}}{(q^3;q^3)_i(q^3;q^3)_j}, \\
    T_r(q):=\sum_{\begin{smallmatrix}
        i,j\geq 0 \\ i-j\equiv r \!\!\! \pmod{3}
    \end{smallmatrix}} \frac{q^{2i^2+2ij+2j^2-2i-j}}{(q^3;q^3)_i(q^3;q^3)_j}.
\end{align}
We have
\begin{align}
  &  S_0(q)=\frac{J_{21,45}-q^3J_{6,45}}{J_3}, \label{11-S0-result}\\
   & S_1(q)=S_{-1}(q)=q^2\frac{J_{9,45}}{J_3}, \label{11-S1-result} \\
&T_0(q)+T_1(q)=2\frac{J_{18,45}}{J_3}, \label{11-T0T1-result} \\
&T_{-1}(q)=\frac{qJ_{12,45}+q^4J_{3,45}}{J_3}. \label{11-T2-result}
\end{align}
\end{lemma}
\begin{proof}
Interchanging $i$ with $j$ it is easy to see that $S_1(q)=S_{-1}(q)$.

Note that
\begin{align}
&2i^2+2ij+2j^2=2(i-j)^2+6ij \nonumber \\
&\equiv 2(i-j)^2 \equiv \left\{\begin{array}{ll}
0 \pmod{3} & i-j\equiv 0 \pmod{3},\\
-1 \pmod{3} & i-j\equiv 1,-1 \pmod{3}. \end{array}
\right.
\end{align}
From \eqref{conj-10-2} we obtain \eqref{11-S0-result} and \eqref{11-S1-result}.

Note that
\begin{align}
    &2i^2+2ij+2j^2-2i-j=2(i-j)^2+(i-j)+6ij-3i \nonumber \\
    &\equiv 2(i-j)^2+(i-j)\equiv \left\{\begin{array}{ll}
0 \pmod{3} & i-j\equiv 0,1 \pmod{3},\\
1 \pmod{3} & i-j\equiv -1 \pmod{3}. \end{array}
\right.
\end{align}
From \eqref{conj-10-1} we obtain   \eqref{11-T0T1-result} and \eqref{11-T2-result}.
\end{proof}

\begin{proof}[Proof of Theorem \ref{thm-lift-11}]
We define
\begin{align}
F(u,v,w;q^4):=\sum_{i,j,k\geq 0} \frac{u^iv^jw^kq^{3i^2+3j^2+3k^2-2ij-2ik-2jk}}{(q^4;q^4)_i(q^4;q^4)_j(q^4;q^4)_k}.
\end{align}
By \eqref{JTP} and \eqref{Euler1} we have
\begin{align}\label{integration2}
&F(u,v,w;q^4)=
\sum_{i,j,k\geq 0} \frac{u^iv^jw^kq^{(i-j-k)^2+2(j-k)^2+2i^2}}{(q^4;q^4)_i(q^4;q^4)_j(q^4;q^4)_k} \nonumber\\
&=\mathrm{CT}_y \mathrm{CT}_z \sum_{i\geq 0}\frac{(uz)^iq^{2i^2}}{(q^4;q^4)_i}\sum_{j\geq 0}\frac{(vy/z)^j}{(q^4;q^4)_j}\sum_{k\geq 0}\frac{(w/yz)^k}{(q^4;q^4)_k}\sum_{m=-\infty}^{\infty}z^{-m}q^{m^2}\sum_{n=-\infty}^{\infty}y^{-n}q^{2n^2} \nonumber\\
&=\mathrm{CT}_y\mathrm{CT}_z \frac{(-q^2uz;q^4)_{\infty}(-qz,-q/z,q^2;q^2)_{\infty}(-q^2y,-q^2/y,q^4;q^4)_{\infty}}{(vy/z,w/yz;q^4)_{\infty}} \nonumber\\
&= \mathrm{CT}_z (-q^2uz,q^4;q^4)_{\infty}(-qz,-q/z,q^2;q^2)_{\infty} \nonumber\\
&\quad \times \mathrm{CT}_y \sum_{i\geq 0}\frac{(-q^2z/v;q^4)_i}{(q^4;q^4)_i}(vy/z)^i\sum_{j\geq 0}\frac{(-q^2z/w;q^4)_j}{(q^4;q^4)_j}(w/yz)^j \quad \text{(by (\ref{q-binomial}))} \nonumber\\
&=\mathrm{CT}_z (-q^2uz,q^4;q^4)_{\infty}(-qz,-q/z,q^2;q^2)_{\infty}\sum_{i\geq 0}\frac{(-q^2z/v,-q^2z/w;q^4)_i}{(q^4,q^4;q^4)_i}(\frac{vw}{z^2})^i  \nonumber\\
&=\mathrm{CT}_z (-q^2uz,q^4;q^4)_{\infty}(-qz,-q/z,q^2;q^2)_{\infty}\frac{(-q^2v/z,-q^2w/z;q^4)_{\infty}}{(vw/z^2,q^4;q^4)_{\infty}} \quad \text{(by (\ref{Gauss}))} \nonumber\\
&=\mathrm{CT}_z \frac{(-q^2uz,-q^2v/z,-q^2w/z;q^4)_{\infty}(-qz,-q/z,q^2;q^2)_{\infty}}{(vw/z^2;q^4)_{\infty}}.
\end{align}

(1) By \eqref{integration2} we have
\begin{align}
&(q^4;q^4)_\infty F(1,1,1;q^4) \nonumber \\
&=\mathrm{CT}_z \frac{(-q^2/z;q^4)_\infty}{(1/z^2;q^4)_{\infty}}(-q^2z,-q^2/z,q^4;q^4)_{\infty}(-qz,-q/z,q^2;q^2)_{\infty} \nonumber\\
&= \mathrm{CT}_z \sum_{i=0}^\infty \frac{z^{-i}q^{2i^2}}{(q^4;q^4)_i} \sum_{j=0}^\infty \frac{z^{-2j}}{(q^4;q^4)_j} \sum_{m=-\infty}^\infty z^{-m}q^{2m^2} \sum_{n=-\infty}^\infty z^nq^{n^2} \quad \nonumber \\
&\qquad \qquad \qquad \qquad \qquad \qquad \text{(by \eqref{JTP} and \eqref{Euler1})}  \nonumber \\
&=\sum_{i,j\geq 0} \frac{q^{2i^2}}{(q^4;q^4)_i(q^4;q^4)_j} \sum_{m=-\infty}^\infty q^{2m^2+(m+i+2j)^2} \nonumber \\
&=\sum_{i,j\geq 0} \frac{q^{2i^2+(i+2j)^2}}{(q^4;q^4)_i(q^4;q^4)_j} \sum_{m=-\infty}^\infty q^{3m^2+2m(i+2j)} \nonumber \\
&=\sum_{i,j\geq 0} \frac{q^{\frac{4}{3}(2i^2+2ij+2j^2)}}{(q^4;q^4)_i(q^4;q^4)_j} \sum_{m=-\infty}^\infty q^{3(m+\frac{1}{3}(i+2j))^2} \label{11-key-step-1} \\
&=\sum_{r=-1}^ 1 \sum_{n\geq 0} \sum_{\begin{smallmatrix}
    i,j\geq 0 \\ i+2j=3n+r
\end{smallmatrix}}  \frac{q^{\frac{4}{3}(2i^2+2ij+2j^2)}}{(q^4;q^4)_i(q^4;q^4)_j} \sum_{m=-\infty}^\infty q^{3(m+\frac{1}{3}(3n+r))^2} \nonumber \\
&=\sum_{r=-1}^ 1 \sum_{n\geq 0} \sum_{\begin{smallmatrix}
    i,j\geq 0 \\ i+2j=3n+r
\end{smallmatrix}}  \frac{q^{\frac{4}{3}(2i^2+2ij+2j^2)}}{(q^4;q^4)_i(q^4;q^4)_j} \sum_{m=-\infty}^\infty q^{3(m+\frac{1}{3}r)^2} \nonumber \\
&=\sum_{r=-1}^ 1 q^{\frac{1}{3}r^2} \sum_{\begin{smallmatrix}
    i,j\geq 0 \\ i-j\equiv r \!\!\!\! \pmod{3}
\end{smallmatrix}}  \frac{q^{\frac{4}{3}(2i^2+2ij+2j^2)}}{(q^4;q^4)_i(q^4;q^4)_j} \sum_{m=-\infty}^\infty q^{3m^2+2mr} \nonumber \\
&=\sum_{r=-1}^ 1 q^{\frac{1}{3}r^2} (-q^{3-2r},-q^{3+2r},q^6;q^6)_\infty \sum_{\begin{smallmatrix}
    i,j\geq 0 \\ i-j\equiv r \!\!\!\! \pmod{3}
\end{smallmatrix}}  \frac{q^{\frac{4}{3}(2i^2+2ij+2j^2)}}{(q^4;q^4)_i(q^4;q^4)_j}  \nonumber \\
&=(-q^3,-q^3,q^6;q^6)_\infty S_0(q^{\frac{4}{3}}) +q^{\frac{1}{3}}(-q,-q^5,q^6;q^6)_\infty (S_1(q^{\frac{4}{3}})+S_{-1}(q^{\frac{4}{3}})). \label{11-proof-1}
\end{align}
Here for the last second equality we used \eqref{JTP}. Substituting \eqref{11-S0-result} and \eqref{11-S1-result} with $q$ replaced by $q^{4/3}$ into \eqref{11-proof-1}, we obtain \eqref{eq-thm-11-1}.

(2)  By \eqref{integration2} we have
\begin{align}
&(q^4;q^4)_\infty F(q^{-2},q^{2},q^{-2};q^4) \nonumber \\
&=\mathrm{CT}_z \frac{(-1/z;q^4)_\infty}{(1/z^2;q^4)_{\infty}}(-z,-q^4/z,q^4;q^4)_{\infty}(-qz,-q/z,q^2;q^2)_{\infty} \nonumber\\
&= \mathrm{CT}_z \sum_{i=0}^\infty \frac{z^{-i}q^{2i^2-2i}}{(q^4;q^4)_i} \sum_{j=0}^\infty \frac{z^{-2j}}{(q^4;q^4)_j} \sum_{m=-\infty}^\infty z^{-m}q^{2m^2+2m} \sum_{n=-\infty}^\infty z^nq^{n^2}  \quad \nonumber \\
&\qquad \qquad \qquad \qquad \qquad \qquad \text{(by \eqref{JTP} and \eqref{Euler1})}  \nonumber \\
&=\sum_{i,j\geq 0} \frac{q^{2i^2-2i}}{(q^4;q^4)_i(q^4;q^4)_j} \sum_{m=-\infty}^\infty q^{2m^2+2m+(m+i+2j)^2} \nonumber \\
&=\sum_{i,j\geq 0} \frac{q^{2i^2-2i+(i+2j)^2}}{(q^4;q^4)_i(q^4;q^4)_j} \sum_{m=-\infty}^\infty q^{3m^2+2m(i+2j+1)} \nonumber \\
&=q^{-\frac{1}{3}}\sum_{i,j\geq 0} \frac{q^{\frac{4}{3}(2i^2+2ij+2j^2-2i-j)}}{(q^4;q^4)_i(q^4;q^4)_j} \sum_{m=-\infty}^\infty q^{3(m+\frac{1}{3}(i+2j+1))^2} \label{11-key-step-2} \\
&=q^{-\frac{1}{3}}\sum_{r=-1}^ 1 \sum_{n\geq 0} \sum_{\begin{smallmatrix}
    i,j\geq 0 \\ i+2j=3n+r-1
\end{smallmatrix}}  \frac{q^{\frac{4}{3}(2i^2+2ij+2j^2-2i-j)}}{(q^4;q^4)_i(q^4;q^4)_j} \sum_{m=-\infty}^\infty q^{3(m+\frac{1}{3}(3n+r))^2} \nonumber \\
&=q^{-\frac{1}{3}}\sum_{r=-1}^ 1 \sum_{n\geq 0} \sum_{\begin{smallmatrix}
    i,j\geq 0 \\ i+2j=3n+r-1
\end{smallmatrix}}  \frac{q^{\frac{4}{3}(2i^2+2ij+2j^2-2i-j)}}{(q^4;q^4)_i(q^4;q^4)_j} \sum_{m=-\infty}^\infty q^{3(m+\frac{1}{3}r)^2} \nonumber \\
&=q^{-\frac{1}{3}}\sum_{r=-1}^ 1 q^{\frac{1}{3}r^2} \sum_{\begin{smallmatrix}
    i,j\geq 0 \\ i-j\equiv r-1 \!\!\!\! \pmod{3}
\end{smallmatrix}}  \frac{q^{\frac{4}{3}(2i^2+2ij+2j^2-2i-j)}}{(q^4;q^4)_i(q^4;q^4)_j} \sum_{m=-\infty}^\infty q^{3m^2+2mr} \nonumber \\
&=q^{-\frac{1}{3}}\sum_{r=-1}^ 1 q^{\frac{1}{3}r^2} (-q^{3-2r},-q^{3+2r},q^6;q^6)_\infty \sum_{\begin{smallmatrix}
    i,j\geq 0 \\ i-j\equiv r-1 \!\!\!\! \pmod{3}
\end{smallmatrix}}  \frac{q^{\frac{4}{3}(2i^2+2ij+2j^2-2i-j)}}{(q^4;q^4)_i(q^4;q^4)_j}   \nonumber \\
&=q^{-\frac{1}{3}}(-q^3,-q^3,q^6;q^6)_\infty T_{-1}(q^{\frac{4}{3}}) +(-q,-q^5,q^6;q^6)_\infty (T_0(q^{\frac{4}{3}})+T_{1}(q^{\frac{4}{3}})). \label{11-proof-2}
\end{align}
Here for the last second equality we used \eqref{JTP}.
Substituting \eqref{11-T0T1-result} and \eqref{11-T2-result} with $q$ replaced by $q^{4/3}$ into \eqref{11-proof-2}, we obtain \eqref{eq-thm-11-2}.
\end{proof}
We remark here that the involvement of  \eqref{conj-10-2} and \eqref{conj-10-1} is totally unexpected before we arrive at \eqref{11-key-step-1} and \eqref{11-key-step-2}.

\section{Applications and Some Other Nahm Sums}\label{sec-applictaion}
Section \ref{sec-exam1} provides two general matrices $\widetilde{A}$ and $\widetilde{A}^\star$  for which four modular triples can be found (see \eqref{exam1-lift} and \eqref{A-exam1-lift-dual}). Besides the data listed there, if we set some specific values for $a$, we may find some extra vectors and scalars so that they  form  modular triples with these matrices. For instance, if we set $a=2$ in \eqref{exam1-lift}, we obtain the seventh matrix in Zagier's rank three examples \cite[Table 3]{Zagier}. There are seven choices of vectors $B$ so that $(\widetilde{A},B,C)$ is modular for some suitable $C$. See \cite[Theorem 4.8]{Wang2024} for the corresponding identities. We will not consider its dual here since the dual examples of all of Zagier's rank three examples will be discussed in a forthcoming paper as promised in \cite{Wang-rank3}. Instead, we give another specific interesting example corresponding to $a=3/2$ to illustrate the phenomenon.

\subsection{A special matrix and the corresponding Nahm sums}
If we set $a=3/2$ in \eqref{eq-exam1-original} and \eqref{exam1-lift}, we obtain
\begin{align}\label{eq-A-32}
A=\begin{pmatrix}
    3/2 & -1/2 \\ -1/2 & 3/2
\end{pmatrix}, \quad \widetilde{A}=\begin{pmatrix} 3/2 & 1/2 & 1 \\  1/2 & 3/2 & 1 \\ 1 & 1 & 2 \end{pmatrix}.
\end{align}
Let
\begin{align}\label{F-defn-special}
F(u,v,w;q^4):=\sum_{i,j,k\geq 0} \frac{u^iv^jw^kq^{3i^2+3j^2+4k^2+2ij+4ik+4jk}}{(q^4;q^4)_i(q^4;q^4)_j(q^4;q^4)_k}.
\end{align}
Recall the identity \eqref{eq-lift-id}, we have
\begin{align}
F(q^{b_1},q^{b_2},q^{b_1+b_2};q^4)=\sum_{i,j\geq 0} \frac{q^{3i^2-2ij+3j^2+b_1i+b_2j}}{(q^4;q^4)_i(q^4;q^4)_j}. \label{eq-thm-32-relation}
\end{align}
The modular triples in \eqref{exam1-lift} with $a=3/2$ give the first four modular triples  $(\widetilde{A},B,C)$ in Table \ref{tab:32-triple}.
\begin{table}[htbp]
    \centering
    \begin{tabular}{ccccccc}
    \hline
        $B$ & $\begin{pmatrix} c \\ -c \\ 0 \end{pmatrix}$ & $\begin{pmatrix} -1/2 \\ -1/2 \\ -1 \end{pmatrix}$ & $\begin{pmatrix} 1/4 \\ 3/4 \\ 1 \end{pmatrix}$ & $\begin{pmatrix} 3/4 \\ 1/4 \\ 1 \end{pmatrix}$  & $\begin{pmatrix} -1/4 \\ -3/4 \\ -1/2 \end{pmatrix}$ & $\begin{pmatrix} -3/4 \\ -1/4 \\ -1/2 \end{pmatrix}$  \\
        $C$ & ${c^2}/{3}-{1}/{24}$ & ${1}/{24}$   & ${7}/{48}$ & ${7}/{48}$   & ${1}/{48}$  &  $1/48$   \\
        \hline
        $B$ &  $\begin{pmatrix} 0 \\ -1/2 \\ 0 \end{pmatrix}$ &  $\begin{pmatrix} -1/2 \\ 0 \\ 0 \end{pmatrix}$ & $\begin{pmatrix} 1/4 \\ -1/4 \\ 1/2 \end{pmatrix}$ & $\begin{pmatrix} -1/4 \\ 1/4 \\ 1/2 \end{pmatrix}$ & $\begin{pmatrix} 1/2 \\ 1/2 \\ 0\end{pmatrix}$ &
$\begin{pmatrix} 1 \\ 1 \\ 1\end{pmatrix}$ \\
        $C$ & $-1/96$  & $-1/96$  & $1/48$ & $1/48$  &  $1/24$  & $7/24$   \\
        \hline
    \end{tabular}
    \caption{Modular triples associated with the matrix $\widetilde{A}$ in \eqref{eq-A-32}.}
    \label{tab:32-triple}
\end{table}

The corresponding Nahm sum identities associated with the first three triples inherited from \eqref{VZ-id-1}--\eqref{VZ-id-3} are
\begin{align}
     \sum_{i,j,k\geq 0} \frac{q^{3i^2+3j^2+4k^2+2ij+4ik+4jk+ci-cj}}{(q^4;q^4)_i(q^4;q^4)_j(q^4;q^4)_k}&=\frac{(-q^{3+c},-q^{3-c},q^6;q^6)_\infty}{(q^4;q^4)_\infty}, \label{lift-VZ-1} \\
     \sum_{i,j,k\geq 0} \frac{q^{3i^2+3j^2+4k^2+4ij+4ik+2jk-2i-2j-4k}}{(q^4;q^4)_i(q^4;q^4)_j(q^4;q^4)_k}&=2\frac{(-q,-q^5,q^6;q^6)_\infty}{(q^4;q^4)_\infty}, \label{lift-VZ-2} \\
      \sum_{i,j,k\geq 0} \frac{q^{3i^2+3j^2+4k^2+4ij+4ik+2jk+i+3j+4k}}{(q^4;q^4)_i(q^4;q^4)_j(q^4;q^4)_k}&=\frac{(-q^6,-q^6,q^6;q^6)_\infty}{(q^4;q^4)_\infty}. \label{lift-VZ-3}
\end{align}
Interchanging $i$ with $j$ in \eqref{lift-VZ-3} yields the identity for the fourth triple.

 As recorded in Table \ref{tab:32-triple}, we find eight more possible modular triples. Due to the symmetry of $i$ and $j$ in the quadratic form associated with $A$, there are essentially five different Nahm sums to be considered. The modularity of three of them follow from the cases $u=1,q^{-1},q$ of the following theorem.
\begin{theorem}\label{thm-2}
We have
\begin{align}
   \sum_{i,j,k\geq 0} \frac{q^{3i^2+3j^2+4k^2+2ij+4ik+4jk-2j}u^{i+j+2k}}{(q^4;q^4)_i(q^4;q^4)_j(q^4;q^4)_k}= (-uq;q^2)_{\infty}.
\end{align}
\end{theorem}
\begin{proof}
By \eqref{Euler1} and \eqref{JTP} we have
\begin{align}\label{integration3}
&F(u,v,w;q^4)=
\sum_{i,j,k\geq 0} \frac{u^iv^jw^kq^{2i^2+2j^2+(i+j+2k)^2}}{(q^4;q^4)_i(q^4;q^4)_j(q^4;q^4)_k} \nonumber \\
&=\mathrm{CT}_z \sum_{i\geq 0}\frac{(uz)^iq^{2i^2}}{(q^4;q^4)_i} \sum_{j\geq 0}\frac{(vz)^jq^{2j^2}}{(q^4;q^4)_j}\sum_{k\geq 0}\frac{(wz^2)^k}{(q^4;q^4)_k}\sum_{m=-\infty}^{\infty}z^{-m}q^{m^2} \nonumber \\
&=\mathrm{CT}_z \frac{(-q^2uz,-q^2vz;q^4)_{\infty}(-qz,-q/z,q^2;q^2)_{\infty}}{(wz^2;q^4)_{\infty}}.
\end{align}
By \eqref{integration3} we have
\begin{align*}
&F(u,q^{-2}u,u^{2};q^4)=\mathrm{CT}_z \frac{(-q^2uz,-uz;q^4)_{\infty}(-qz,-q/z,q^2;q^2)_{\infty}}{(u^2z^2;q^4)_{\infty}}\nonumber \\
&=\mathrm{CT}_z \frac{(-uz;q^2)_{\infty}(-qz,-q/z,q^2;q^2)_{\infty}}{(u^2z^2;q^4)_{\infty}}
=\mathrm{CT}_z \frac{(-qz,-q/z,q^2;q^2)_{\infty}}{(uz;q^2)_{\infty}}  \nonumber \\
&=\mathrm{CT}_z \sum_{i=0}^\infty \frac{u^iz^i}{(q^2;q^2)_i} \sum_{m=-\infty}^{\infty}z^{-m}q^{m^2}  \nonumber \\
&=\sum_{i=0}^\infty \frac{u^iq^{i^2}}{(q^2;q^2)_i}
=(-uq;q^2)_{\infty}.  \quad \text{(by \eqref{Euler1})} \qedhere
\end{align*}
\end{proof}

The last two choices of $(B,C)$ in Table \ref{tab:32-triple} were found by applying the dual operator to the Nahm sums in Theorem \ref{thm-2-Ex1-3/2-in} below. Zagier's duality conjecture  implies that  $q^{1/6}F(q^2,q^2,1;q^4)$ and $q^{7/6}F(q^4,q^4,q^4;q^4)$ are likely to be modular. However, we find that they can be expressed as sums of two modular forms of weights 0 and 1, respectively. Hence $q^{C}F(q^2,q^2,1;q^4)$ and $q^{C}F(q^4,q^4,q^4;q^4)$  are not modular for any $C$.
\begin{theorem}\label{thm-nonmodular}
We have
\begin{align}
   &\sum_{i,j,k\geq 0} \frac{q^{3i^2+3j^2+4k^2+2ij+4ik+4jk+2i+2j}}{(q^4;q^4)_i(q^4;q^4)_j(q^4;q^4)_k}= \frac{1}{3}\frac{J_1^2J_4}{J_2}+\frac{2}{3}\frac{J_2^2J_3J_{12}}{J_1J_4^2J_6}, \label{nonmodular-id-1}\\
   &\sum_{i,j,k\geq 0} \frac{q^{3i^2+3j^2+4k^2+2ij+4ik+4jk+4i+4j+4k}}{(q^4;q^4)_i(q^4;q^4)_j(q^4;q^4)_k}= -\frac{1}{3}q^{-1}\frac{J_1^2J_4}{J_2}+\frac{1}{3}q^{-1}\frac{J_2^2J_3J_{12}}{J_1J_4^2J_6}. \label{nonmodular-id-2}
\end{align}
\end{theorem}
\begin{proof}
From the definition \eqref{F-defn-special} we have
\begin{align}\label{eq-proof-nonmodular-F}
&F(u,u,w;q^4)=\sum_{n,k\geq 0} \sum_{j=0}^n \frac{u^nw^kq^{2n^2+(n-2j)^2+4k^2+4nk}}{(q^4;q^4)_{n-j}(q^4;q^4)_j(q^4;q^4)_k} \nonumber \\
&=S_0(u,w;q^4)+S_1(u,w;q^4).
\end{align}
Here $S_0(u,w;q^4)$ and $S_1(u,w;q^4)$ correspond to the sum with even and odd values of $n$, respectively. That is,
\begin{align}
&S_0(u,w;q^4)=\sum_{n,k\geq 0} \sum_{j=0}^{2n} \frac{u^{2n}w^kq^{8n^2+4(n-j)^2+4k^2+8nk}}{(q^4;q^4)_{2n-j}(q^4;q^4)_j(q^4;q^4)_k}, \\
&S_1(u,w;q^4)=\sum_{n,k\geq 0} \sum_{j=0}^{2n+1} \frac{u^{2n+1}w^kq^{8n^2+8n+4(n-j)^2+4(n-j)+4k^2+8nk+4k+3}}{(q^4;q^4)_{2n-j+1}(q^4;q^4)_j(q^4;q^4)_k}.
\end{align}
We have
\begin{align}
&S_0(u,w;q^4)=\sum_{n,k\geq 0} \sum_{j=0}^{2n} \frac{u^{2n}w^kq^{4n^2+4(n-j)^2+4(n+k)^2}}{(q^4;q^4)_{2n-j}(q^4;q^4)_j(q^4;q^4)_k} \nonumber \\
&=\sum_{n,k\geq 0} \frac{u^{2n}w^kq^{4n^2+4(n+k)^2}}{(q^4;q^4)_k}\Big(\frac{1}{(q^4;q^4)_n^2}+2\sum_{r=1}^n \frac{q^{4r^2}}{(q^4;q^4)_{n-r}(q^4;q^4)_{n+r}}  \Big) \nonumber \\
&=\sum_{m=0}^\infty \sum_{n=0}^m \frac{u^{2n}w^{m-n}q^{4n^2+4m^2}}{(q^4;q^4)_{m-n}} \Big(\frac{1}{(q^4;q^4)_n^2}+2\sum_{r=1}^n \frac{q^{4r^2}}{(q^4;q^4)_{n-r}(q^4;q^4)_{n+r}}\Big). \label{nonmodular-S0-start}
\end{align}
Similarly, we have
\begin{align}
&S_1(u,w;q^4)=\sum_{n,k\geq 0} \sum_{j=0}^{2n+1} \frac{u^{2n+1}w^kq^{4n^2+4n+4(n-j)^2+4(n-j)+4(n+k)^2+4(n+k)+3}}{(q^4;q^4)_{2n-j+1}(q^4;q^4)_j(q^4;q^4)_k} \nonumber \\
&=2\sum_{n,k\geq 0} \frac{u^{2n+1}w^kq^{4n^2+4n+3+4(n+k)^2+4(n+k)}}{(q^4;q^4)_k}\sum_{r=0}^n \frac{q^{4r^2+4r}}{(q^4;q^4)_{n-r}(q^4;q^4)_{n+r+1}} \nonumber \\
&=2\sum_{m=0}^\infty \sum_{n=0}^m \frac{u^{2n+1}w^{m-n}q^{4n^2+4n+3+4m^2+4m}}{(q^4;q^4)_{m-n}} \sum_{r=0}^n \frac{q^{4r^2+4r}}{(q^4;q^4)_{n-r}(q^4;q^4)_{n+r+1}}. \label{nonmodular-S1}
\end{align}
Letting $u=q^2$ and $w=1$ and then replacing $q$ by $q^{1/4}$, we have
\begin{align}
&S_0(q^{1/2},1;q)=\sum_{m=0}^\infty q^{m^2} \sum_{n=0}^m \frac{q^{n^2+n}}{(q;q)_{m-n}} \Big(\frac{1}{(q;q)_n^2}+2\sum_{r=1}^n \frac{q^{r^2}}{(q;q)_{n-r}(q;q)_{n+r}} \Big), \label{add-S0-start} \\
&S_1(q^{1/2},1;q)=2q^{5/4}\sum_{m=0}^\infty q^{m^2+m} \sum_{n=0}^m \frac{q^{n^2+2n}}{(q;q)_{m-n}} \sum_{r=0}^n \frac{q^{r^2+r}}{(q;q)_{n-r}(q;q)_{n+r+1}}. \label{add-S1-start}
\end{align}

In order to calculate $S_0(q^{1/2},1;q)$ we define a Bailey pair $(\alpha_n^{(0)}(1;q),\beta_n^{(0)}(1;q))$ with
\begin{align}
\alpha_n^{(0)}(1;q):=\left\{\begin{array}{ll}
1 & n=0, \\
2q^{n^2} & n\geq 1.
\end{array}\right.
\end{align}
Using \eqref{eq-BP-lift} we obtain a Bailey pair $(\alpha_n^{(1)}(q;q),\beta_n^{(1)}(q;q)$ where
\begin{align}\label{S0-BP-1}
\alpha_n^{(1)}(q;q)=(2n+1)\frac{q^{n^2}(1-q^{2n+1})}{1-q}.
\end{align}
Applying \eqref{BP-S1} we obtain the Bailey pair
\begin{align}\label{S0-BP-2}
\alpha_n^{(2)}(q;q)=(2n+1)\frac{q^{2n^2+n}(1-q^{2n+1})}{1-q}, \quad \beta_n^{(2)}(q;q)=\sum_{k=0}^n \frac{q^{k^2+k}}{(q;q)_{n-k}}\beta_k^{(1)}(q;q).
\end{align}
Next, applying Lemma \ref{lem-BP-down} we obtain the Bailey pair
\begin{equation}\label{S0-BP-3}
\begin{split}
\alpha_0^{(3)}(1;q)&=1, ~~ \alpha_n^{(3)}(1;q)=(2n+1)q^{2n^2+n}-(2n-1)q^{2n^2-n} ~~ (n\geq 1), \\
\beta_n^{(3)}(1;q)&=\beta_n^{(2)}(1;q).
\end{split}
\end{equation}
Using \eqref{eq-BP-id-key} with the Bailey pairs $(\alpha_n^{(i)};\beta_n^{(i)})$ ($i=0,1,2,3$) in \eqref{add-S0-start}, we deduce that
\begin{align}
&S_0(q^{1/2},1;q)=\sum_{m=0}^\infty q^{m^2} \sum_{n=0}^m \frac{q^{n^2+n}}{(q;q)_{m-n}} \beta_m^{(0)}(1;q) \nonumber \\
&=\sum_{m=0}^\infty q^{m^2} \beta_m^{(2)}(q;q) =\sum_{m=0}^\infty q^{m^2} \beta_m^{(3)}(1;q) \nonumber \\
&=\frac{1}{(q;q)_\infty} \sum_{n=0}^\infty q^{n^2} \alpha_n^{(3)}(1;q) \nonumber \\
&=\frac{1}{(q;q)_\infty} \Big(1+\sum_{n=1}^\infty \big((2n+1)q^{3n^2+n}-(2n-1)q^{3n^2-n}\big)\Big) \nonumber \\
&=\frac{1}{(q;q)_\infty} \sum_{n=-\infty}^\infty (2n+1)q^{3n^2+n}. \label{nonmodular-S0-result}
\end{align}

In order to calculate $S_1(q^{1/2},1;q)$ we define a Bailey pair $(\widetilde{\alpha}_n^{(0)}(q;q),\widetilde{\beta}_n^{(0)}(q;q))$ with
\begin{align}
\widetilde{\alpha}_n^{(0)}(q;q)=q^{n^2+n}, \quad n\geq 0.
\end{align}
Using \eqref{eq-BP-lift} we obtain the Bailey pair
\begin{align}
\widetilde{\alpha}_n^{(1)}(q^2;q)=(n+1)\frac{(1-q^{2n+2})q^{n^2+n}}{1-q^2}, \quad \widetilde{\beta}_n^{(1)}(q^2;q)=\widetilde{\beta}_n^{(0)}(q;q).
\end{align}
Using  \eqref{BP-S1} we obtain the Bailey pair
\begin{equation}
\begin{split}
\widetilde{\alpha}_n^{(2)}(q^2;q)&=(n+1)\frac{q^{2n^2+3n}(1-q^{2n+2})}{1-q^2}, \\
 \widetilde{\beta}_n^{(2)}(q^2;q)&=\sum_{k=0}^n \frac{q^{k^2+2k}}{(q;q)_{n-k}}\widetilde{\beta}_k^{(1)}(q^2;q).
\end{split}
\end{equation}
Next, using Lemma \ref{lem-BP-down} we obtain the Bailey pair
\begin{align}
\widetilde{\alpha}_n^{(3)}(q;q)=(n+1)q^{2n^2+3n}-nq^{2n^2+n-1}, \quad \widetilde{\beta}_n^{(3)}(q;q)=\beta_n^{(2)}(q^2;q).
\end{align}
Using \eqref{eq-BP-id-key} with the Bailey pairs $(\widetilde{\alpha}_n^{(i)};\widetilde{\beta}_n^{(i)})$ ($i=0,1,2,3$) in \eqref{add-S1-start}, we deduce that
\begin{align}
&S_1(q^{1/2},1;q)=2\frac{q^{5/4}}{1-q} \sum_{m=0}^\infty q^{m^2+m} \sum_{n=0}^m \frac{q^{n^2+2n}}{(q;q)_{m-n}} \widetilde{\beta}_n^{(0)}(q;q) \nonumber \\
&=2\frac{q^{5/4}}{1-q} \sum_{m=0}^\infty q^{m^2+m} \widetilde{\beta}_m^{(2)}(q^2;q)=2\frac{q^{5/4}}{1-q} \sum_{m=0}^\infty q^{m^2+m}\widetilde{\beta}_m^{(3)}(q;q) \nonumber \\
&=2\frac{q^{5/4}}{(q;q)_\infty} \sum_{n=0}^\infty q^{n^2+n}\widetilde{\alpha}_n^{(3)}(q;q) \nonumber \\
&=2\frac{q^{5/4}}{(q;q)_\infty} \sum_{n=0}^\infty \Big( (n+1)q^{3n^2+4n}-nq^{3n^2+2n-1}\Big) \nonumber \\
&=-2\frac{q^{1/4}}{(q;q)_\infty} \sum_{n=-\infty}^\infty nq^{3n^2+2n}. \label{nonmodular-S1-result}
\end{align}
Substituting \eqref{nonmodular-S0-result} and \eqref{nonmodular-S1-result} with $q$ replaced by $q^4$ into \eqref{eq-proof-nonmodular-F}, we deduce that
\begin{align}
&F(q^2,q^2,1;q^4)=\frac{1}{(q^4;q^4)_\infty} \Big(\sum_{n=-\infty}^\infty (2n+1)q^{12n^2+4n}-\sum_{n=-\infty}^\infty (2n)q^{12n^2+8n+1}\Big) \nonumber \\
&=\frac{1}{(q^4;q^4)_\infty} \sum_{k=-\infty}^\infty (k+1)q^{3k^2+2k}. \label{nonmodular-F-final}
\end{align}
Recall the following identity from \cite[Eq.\ (2.34)]{Wang2024}:
\begin{align}
\sum_{k=-\infty}^\infty (k+1)q^{3k^2+2k}=\frac{1}{3}\frac{J_1^2J_4^2}{J_2}+\frac{2}{3}\frac{J_2^2J_3J_{12}}{J_1J_4J_6}. \label{Wang-id-1}
\end{align}
Substituting \eqref{Wang-id-1} into \eqref{nonmodular-F-final}, we obtain \eqref{nonmodular-id-1}.

We can prove \eqref{nonmodular-id-2} in a similar way, but here we prefer to use a different approach. Note that
\begin{align}
&F(q^2,q^{-2},1;q^4)-F(q^2,q^2,1;q^4)=\sum_{i,j,k\geq 0} \frac{q^{3i^2+3j^2+4k^2+2ij+4ik+4jk+2i-2j}(1-q^{4j})}{(q^4;q^4)_i(q^4;q^4)_j(q^4;q^4)_k} \nonumber \\
&=\sum_{i,j,k\geq 0} \frac{q^{3i^2+3(j+1)^2+4k^2+2i(j+1)+4ik+4(j+1)k+2i-2(j+1)}}{(q^4;q^4)_i(q^4;q^4)_j(q^4;q^4)_k} =qF(q^4,q^4,q^4;q^4). \label{nonmodular-relation}
\end{align}
Substituting \eqref{lift-VZ-1} and \eqref{nonmodular-id-1} into \eqref{nonmodular-relation}, we obtain \eqref{nonmodular-id-2}.
\end{proof}

\subsection{Dual Nahm sums}
Now we consider the dual examples of Table \ref{tab:32-triple}. We expect that $(A,B,C)$ are modular triples where (set $a=3/2$ in \eqref{A-exam1-lift-dual})
\begin{align}\label{A-exam1-lift-dual-32}
&A=\begin{pmatrix} 1 & 0 & -1/2 \\  0 & 1 & -1/2 \\ -1/2 & -1/2 & 1 \end{pmatrix},
\end{align}
and $B,C$ are given in Table \ref{tab-exam1-lift-dual}. Here the last two vectors $(\frac{1}{2},\frac{1}{2},-\frac{1}{2})^\mathrm{T}$ and $(\frac{1}{2},\frac{1}{2},0)^\mathrm{T}$ were found through a Maple search, and the corresponding values of $C$ were found after finding the identities \eqref{thm3.7-6} and \eqref{thm3.7-7}. These two modular triples lead us to the discovery of the last two choices of $(B,C)$ in Table \ref{tab:32-triple} by the dual operation.

\begin{table}[H]
    \centering
    \begin{tabular}{ccccccc}
    \hline
        $B$ &  $\begin{pmatrix} -c \\ c \\ 0 \end{pmatrix}$  & $\begin{pmatrix} 0 \\ 0 \\ -1/2 \end{pmatrix}$ &  $\begin{pmatrix} -1/4 \\ 1/4 \\ 1/2 \end{pmatrix}$ & $\begin{pmatrix} 1/4 \\ -1/4 \\ 1/2 \end{pmatrix}$  & $\begin{pmatrix} 0 \\ -1/2 \\ 0 \end{pmatrix}$  & $\begin{pmatrix} -1/2 \\ 0 \\ 0 \end{pmatrix}$ \\
        $C$ & $(8c^2-1)/12$   & $1/12$ & $1/24$ &  $1/24$ &  $1/{24}$ & $1/24$      \\
        \hline
        $B$   &  $\begin{pmatrix} 0 \\ -1/2 \\ 1/4 \end{pmatrix}$  &  $\begin{pmatrix} -1/2 \\ 0 \\ 1/4 \end{pmatrix}$   & $\begin{pmatrix} 0 \\ -1/2 \\ 1/2 \end{pmatrix}$  &  $\begin{pmatrix} -1/2 \\ 0 \\ 1/2 \end{pmatrix}$   &  $\begin{pmatrix} 1/2\\ 1/2 \\ -1/2 \end{pmatrix}$ & $\begin{pmatrix} 1/2\\ 1/2 \\ 0  \end{pmatrix}$  \\
        $C$  & $1/96$ & $1/96$  & $1/24$     & $1/24$   &  $1/12$  & $1/12$   \\
        \hline
    \end{tabular}
    \caption{Modular triples associated with the matrix $A$ in \eqref{A-exam1-lift-dual-32}.}
    \label{tab-exam1-lift-dual}
\end{table}

Since $i$ and $j$ are symmetric in the quadratic form generated by $A$, there are essentially eight different Nahm sums to be considered. Their modularity is proved by the following theorem.
\begin{theorem}\label{thm-2-Ex1-3/2-in}
We have
\begin{align}
   & \sum_{i,j,k\geq 0} \frac{q^{2i^2+2j^2+2k^2-2ik-2jk-2j+ak}}{(q^4;q^4)_i(q^4;q^4)_j(q^4;q^4)_k}=(-1,-q^a;q^2)_\infty, \label{thm3.7-1} \\
   &   \sum_{i,j,k\geq 0} \frac{q^{2i^2+2j^2+2k^2-2ik-2jk+i-j+2k}}{(q^4;q^4)_i(q^4;q^4)_j(q^4;q^4)_k}=\frac{J_2^2J_6^2}{J_1J_3J_4^2}, \label{thm3.7-3} \\
   &  \sum_{i,j,k\geq 0} \frac{q^{2i^2+2j^2+2k^2-2ik-2jk-ci+cj}}{(q^4;q^4)_i(q^4;q^4)_j(q^4;q^4)_k}=\frac{\overline{J}_{2+c,4}\overline{J}_{6+c,12}}{J_4^2}+q^2\frac{\overline{J}_{-c,4}\overline{J}_{c,12}}{J_4^2}, \label{thm3.7-4} \\
& \sum_{i,j,k\geq 0} \frac{q^{i^2+j^2+k^2-ik-jk-k}}{(q^2;q^2)_i(q^2;q^2)_j(q^2;q^2)_k}=6\frac{J_3^3}{J_1J_2^2}, \label{thm3.7-5}  \\
&\sum_{i,j,k\geq 0} \frac{q^{i^2+j^2+k^2-ik-jk+i+j-k}}{(q^2;q^2)_i(q^2;q^2)_j(q^2;q^2)_k}=2\frac{J_3^3}{J_1J_2^2}, \label{thm3.7-6} \\
&\sum_{i,j,k\geq 0} \frac{q^{i^2+j^2+k^2-ik-jk+i+j}}{(q^2;q^2)_i(q^2;q^2)_j(q^2;q^2)_k}=\frac{J_3^3}{J_1J_2^2}.\label{thm3.7-7}
\end{align}
\end{theorem}

\begin{proof}
We define
\begin{align}
F(u,v,w;q^4):=\sum_{i,j,k\geq 0} \frac{u^iv^jw^kq^{2i^2+2j^2+2k^2-2ik-2jk}}{(q^4;q^4)_i(q^4;q^4)_j(q^4;q^4)_k}.
\end{align}
By \eqref{Euler1} we have
\begin{align}\label{rank2-in3.30}
&F(u,v,w;q^4)=\sum_{k\geq 0} \frac{w^kq^{2k^2}}{(q^4;q^4)_k}(-uq^{2-2k};q^4)_{\infty}(-vq^{2-2k};q^4)_{\infty} \nonumber \\
&=(-uq^{2},-vq^{2};q^4)_{\infty}\sum_{i\geq 0} \frac{u^{i}v^iw^{2i}q^{4i^2}(-q^2/u,-q^2/v;q^4)_{i}}{(q^4;q^4)_{2i}} \nonumber \\
&\qquad +(-u,-v;q^4)_{\infty}\sum_{i\geq 0} \frac{u^iv^iw^{2i+1}q^{4i^2+4i+2}(-q^4/u,-q^4/v;q^4)_{i}}{(q^4;q^4)_{2i+1}}.
\end{align}

(1) By \eqref{rank2-in3.30} we have
\begin{align}
&F(1,q^{-2},q^a;q^4)=(-1;q^2)_{\infty}\sum_{i\geq 0} \frac{q^{4i^2-2i+2ai}(-q^2;q^2)_{2i}}{(q^4;q^4)_{2i}} \nonumber \\
& \qquad +(-q^{-2};q^2)_{\infty}\sum_{i\geq 0} \frac{q^{4i^2+2i+2ai+a+2}(-q^4;q^2)_{2i}}{(q^4;q^4)_{2i+1}}\nonumber \\
&=(-1;q^2)_{\infty}\Big(\sum_{i\geq 0} \frac{q^{4i^2-2i+2ai}}{(q^2;q^2)_{2i}}+\sum_{i\geq 0} \frac{q^{4i^2+2i+2ai+a}}{(q^2;q^2)_{2i+1}}\Big) \nonumber \\
&=(-1;q^2)_{\infty}\sum_{k\geq 0} \frac{q^{k^2-k+ak}}{(q^2;q^2)_{k}} =(-1,-q^a;q^2)_{\infty}.
\end{align}

(2) By \eqref{rank2-in3.30} we have
\begin{align}
&F(q,q^{-1},q^b;q^4)=(-q;q^2)_{\infty}\sum_{i\geq 0} \frac{q^{4i^2+2bi}(-q;q^2)_{2i}}{(q^4;q^4)_{2i}} \nonumber \\
&\qquad \qquad \qquad \qquad +(-q^{-1};q^2)_{\infty}\sum_{i\geq 0} \frac{q^{4i^2+4i+2bi+b+2}(-q^3;q^2)_{2i}}{(q^4;q^4)_{2i+1}}  \nonumber \\
=&(-q;q^2)_{\infty}\Big(\sum_{i\geq 0} \frac{q^{4i^2+2bi}(-q;q^2)_{2i}}{(q^4;q^4)_{2i}}+\sum_{i\geq 0} \frac{q^{4i^2+4i+2bi+b+1}(-q;q^2)_{2i+1}}{(q^4;q^4)_{2i+1}}\Big) \nonumber \\
=&(-q;q^2)_{\infty}\sum_{k\geq 0} \frac{q^{k^2+bk}(-q;q^2)_{k}}{(q^4;q^4)_{k}}. \label{thm3.7-proof-2}
\end{align}
Setting $b=2$ in \eqref{thm3.7-proof-2} and using \eqref{Entry 4.2.11} with $q$ replaced by $-q$, we obtain \eqref{thm3.7-3}.

(3) By \eqref{rank2-in3.30} we have
\begin{align}\label{proof-dual32-3}
F(q^{-c},q^{c},1;q^4)=(-q^{2+c},-q^{2-c};q^4)_{\infty}S_0(q)+(-q^{c},-q^{-c};q^4)_{\infty}S_1(q).
\end{align}
Here
\begin{align}
&S_0(q)=\sum_{i\geq 0} \frac{q^{4i^2}(-q^{2-c},-q^{2+c};q^4)_{i}}{(q^4;q^4)_{2i}}=\frac{(-q^{6+c},-q^{6-c},q^{12};q^{12})_{\infty}}{(q^4;q^4)_{\infty}} ~~ \text{(by \eqref{Entry 5.3.1})}, \label{proof-32-3-S0}\\
&S_1(q)=\sum_{i\geq 0} \frac{q^{4i^2+4i+2}(-q^{4-c},-q^{4+c};q^4)_{i}}{(q^4;q^4)_{2i+1}} \nonumber \\
&= \frac{q^2}{1+q^c}\sum_{i\geq 0} \frac{q^{4i^2+4i}(-q^{4-c};q^4)_{i}(-q^{c};q^4)_{i+1}}{(q^4;q^4)_{2i+1}} \nonumber \\
&=\frac{q^2}{1+q^c}\lim\limits_{\rho_{1},\rho_{2}\to \infty}\sum_{i\geq 0} \frac{(-q^{4-c},\rho_{1},\rho_{2};q^4)_{i}(-q^{c};q^4)_{i+1}}{(q^4;q^4)_{2i+1}}\left(\frac{q^8}{\rho_{1}\rho_{2}}\right)^i \nonumber \\
&=\frac{q^2}{1+q^c}\times \frac{1}{(q^4;q^4)_{\infty}}\sum_{i\geq 0} (q^{c(i+1)}+q^{-ci})q^{6i^2+6i}
\quad\text{(by \eqref{Part2-5.2.4})}\nonumber \\
&=\frac{q^2}{1+q^c}\times \frac{1}{(q^4;q^4)_{\infty}}\sum_{i=-\infty}^{\infty} q^{6i^2+(6-c)i} \nonumber \\
&=\frac{q^2}{1+q^c}\times \frac{(-q^{c},-q^{12-c},q^{12};q^{12})_{\infty}}{(q^4;q^4)_{\infty}}. \label{proof-32-3-S1}
\end{align}
Substituting \eqref{proof-32-3-S0} and \eqref{proof-32-3-S1} into \eqref{proof-dual32-3}, we obtain \eqref{thm3.7-4}.

(4) By \eqref{rank2-in3.30} we have
\begin{align}
&F(1,1,q^{-1};q^2)=(-q,-q;q^2)_{\infty}\sum_{i\geq 0} \frac{q^{2i^2-2i}(-q,-q;q^2)_{i}}{(q^2;q^2)_{2i}}  \nonumber \\
&\qquad +(-1,-1;q^2)_{\infty}\sum_{i\geq 0} \frac{q^{2i^2}(-q^{2},-q^{2};q^2)_{i}}{(q^2;q^2)_{2i+1}} \nonumber \\
&=\frac{J_2^4}{J_1^2J_4^2}\sum_{i\geq 0} \frac{q^{2i^2-2i}(-q;q^2)_{i}}{(q;q^2)_{i}(q^4;q^4)_i} +4\frac{J_4^3J_6^2}{J_2^4J_{12}}.\label{add-id-1}
\end{align}
Here we used  \eqref{Entry 4.2.8+4.2.9} with $q$ replaced by $q^2$.

From \eqref{proof-32-3-S1} with $c=2$ we deduce that
\begin{align}
    \sum_{i\geq 0} \frac{q^{2i^2+2i}(-q;q^2)_{i}}{(q;q^2)_{i+1}(q^4;q^4)_{i}} =\frac{\overline{J}_{1,6}}{J_2}. \label{add-id-2}
\end{align}
We now show that the first sum in \eqref{add-id-1} is equal to twice the sum in \eqref{add-id-2}.

For each integer $r\geq 0$ we define
\begin{align}\label{PrEntry4.2.9}
f_{r}(q^2):=\sum_{i\geq 0} \frac{q^{2i^2-2i}(-q;q^2)_{i+r}}{(q;q^2)_{i+r}(q^4;q^4)_{i}}
-2\sum_{i\geq 0} \frac{q^{2i^2+2i}(-q;q^2)_{i+r}}{(q;q^2)_{i+r+1}(q^4;q^4)_{i}}.
\end{align}
Hence,
\begin{align}
&f_{r}(q^2)=\sum_{i\geq 0} \frac{q^{2i^2-2i}(-q;q^2)_{i+r}}{(q;q^2)_{i+r+1}(q^4;q^4)_{i}}
((1-q^{2(i+r)+1})-2q^{4i}) \nonumber \\
&=\sum_{i\geq 0} \frac{q^{2i^2-2i}(-q;q^2)_{i+r}}{(q;q^2)_{i+r+1}(q^4;q^4)_{i}}
(2(1-q^{4i})-(1+q^{2(i+r)+1})) \nonumber \\
&=2\sum_{i\geq 1} \frac{q^{2i^2-2i}(-q;q^2)_{i+r}}{(q;q^2)_{i+r+1}(q^4;q^4)_{i-1}}
-\sum_{i\geq 0} \frac{q^{2i^2-2i}(-q;q^2)_{i+r+1}}{(q;q^2)_{i+r+1}(q^4;q^4)_{i}}  \nonumber \\
&=2\sum_{i\geq 0} \frac{q^{2i^2+2i}(-q;q^2)_{i+r+1}}{(q;q^2)_{i+r+2}(q^4;q^4)_{i}}
-\sum_{i\geq 0} \frac{q^{2i^2-2i}(-q;q^2)_{i+r+1}}{(q;q^2)_{i+r+1}(q^4;q^4)_{i}}   \nonumber\\
&=-f_{r+1}(q^2).
\end{align}
Note that
\begin{align}
&\lim\limits_{r \to \infty}f_{r}(q^2)=\frac{(-q;q^2)_{\infty}}{(q;q^2)_{\infty}}\Big(\sum_{i\geq 0} \frac{q^{2i^2-2i}}{(q^4;q^4)_{i}}
-2\sum_{i\geq 0} \frac{q^{2i^2+2i}}{(q^4;q^4)_{i}}\Big)\nonumber\\
&=\frac{(-q;q^2)_{\infty}}{(q;q^2)_{\infty}}\left((-1;q^4)_{\infty}-2(-q^4;q^4)_{\infty} \right)=0.
\end{align}
Therefore, the recurrence formula above implies that $f_{r}(q^2)=0$.
Letting $r=0$ in \eqref{PrEntry4.2.9}, we prove the previous assertion that
\begin{align}
 \sum_{i\geq 0} \frac{q^{2i^2-2i}(-q;q^2)_{i}}{(q;q^2)_i(q^4;q^4)_{i}} =2\frac{\overline{J}_{1,6}}{J_2}. \label{add-id-3}
\end{align}
Substituting \eqref{add-id-3} into \eqref{add-id-1}, we deduce that
\begin{align}
    &F(1,1,q^{-1};q^2)=2\frac{J_2^5J_3J_{12}}{J_1^3J_4^3J_6}+4\frac{J_4^3J_6^2}{J_2^4J_{12}} \nonumber \\
    &=2\frac{J_2^5J_{12}}{J_4^3J_6}\left( \frac{J_4^6J_6^3}{J_2^9J_{12}^2}+3q\frac{J_4^2J_6J_{12}^2}{J_2^7} \right)+4\frac{J_4^3J_6^2}{J_2^4J_{12}} =6\frac{J_3^3}{J_1J_2^2}. \label{add-id-4}
\end{align}
Here for the last two equalities we used the following identities \cite[Eqs.\ (3.75) and (3.38)]{XiaYao}:
\begin{align}
    \frac{J_3^3}{J_1}&=\frac{J_4^3J_6^2}{J_2^2J_{12}}+q\frac{J_{12}^3}{J_4}, \label{3core-id-1} \\
    \frac{J_3}{J_1^3}&=\frac{J_4^6J_6^3}{J_2^9J_{12}^2}+3q\frac{J_4^2J_6J_{12}^2}{J_2^7}. \label{3core-id-2}
\end{align}
This proves \eqref{thm3.7-5}.

(5) Recall the following identities from \cite[Corollary 2.2]{Wang2024}:
\begin{align}
\sum_{n=0}^\infty \frac{q^{n^2}(-1;q)_n^2}{(q;q)_{2n}}&=\frac{4}{3}\frac{J_2J_3^2}{J_1^2J_6}-\frac{1}{3}\frac{J_1^4}{J_2^2}, \label{key-id-1} \\
\sum_{n=0}^\infty \frac{q^{2n^2+2n}(-q;q^2)_n^2}{(q^2;q^2)_{2n+1}}
&=\frac{1}{3}\frac{J_1^2J_4^2}{J_2^2}+\frac{2}{3}\frac{J_2J_3J_{12}}{J_1J_4J_6}. \label{key-id-2}
\end{align}

We have
\begin{align}\label{add-6-F-start}
&F(q,q,q^{-1};q^2)=(-q^2;q^2)_\infty^2 \sum_{n=0}^\infty \frac{q^{2n^2}(-1;q^2)_n^2}{(q^2;q^2)_{2n}}+(-q;q^2)_\infty^2 \sum_{n=0}^\infty \frac{q^{2n^2+2n}(-q;q^2)_n^2}{(q^2;q^2)_{2n+1}} \nonumber \\
&=\frac{J_4^2}{J_2^2}\Big(\frac{4}{3}\frac{J_4J_6^2}{J_2^2J_{12}}-\frac{1}{3}\frac{J_2^4}{J_4^2}\Big)+\frac{J_2^4}{J_1^2J_4^2}\Big( \frac{1}{3}\frac{J_1^2J_4^2}{J_2^2}+\frac{2}{3}\frac{J_2J_3J_{12}}{J_1J_4J_6}\Big) \nonumber \\
&=\frac{4}{3}\frac{J_4^3J_6^2}{J_2^4J_{12}}+\frac{2}{3}\frac{J_2^5J_3J_{12}}{J_1^3J_4^3J_6}=2\frac{J_3^3}{J_1J_2^2}.
\end{align}
Here for the last equality we used \eqref{add-id-4}. This proves \eqref{thm3.7-6}.

(6) Using \eqref{thm3.7-4} with $c=2$ and \eqref{add-id-4} we obtain
\begin{align}
    F(q^{-1},q,1;q^2)=2\frac{J_4^3J_6^2}{J_2^4J_{12}}+\frac{J_2^5J_3J_{12}}{J_1^3J_4^3J_6}=3\frac{J_3^3}{J_1J_2^2}. \label{add-id-5}
\end{align}
It is easy to see that
\begin{align}
    &F(q^{-1},q,1;q^2)-F(q,q,1;q^2)=\sum_{i,j,k\geq 0} \frac{q^{i^2+j^2+k^2-ik-jk-i+j}(1-q^{2i})}{(q^2;q^2)_i(q^2;q^2)_j(q^2;q^2)_k} \nonumber \\
 &=\sum_{i,j,k\geq 0} \frac{q^{(i+1)^2+j^2+k^2-(i+j+1)k+j-i-1}}{(q^2;q^2)_i(q^2;q^2)_j(q^2;q^2)_k} \nonumber \\
 &=\sum_{i,j,k\geq 0} \frac{q^{i^2+j^2+k^2-ik-jk+i+j-k}}{(q^2;q^2)_i(q^2;q^2)_j(q^2;q^2)_k}=F(q,q,q^{-1};q^2). \label{thm3.7-proof-5}
\end{align}
Substituting \eqref{add-6-F-start} and \eqref{add-id-5} into \eqref{thm3.7-proof-5}, we obtain \eqref{thm3.7-7}.
\end{proof}
\begin{rem}
If we set $b=0$ in \eqref{thm3.7-proof-2} and using \eqref{S. 25}, we deduce that
\begin{align}
F(q,q^{-1},1;q^4)=(-q;q^2)_{\infty}\sum_{k\geq 0} \frac{q^{k^2}(-q;q^2)_{k}}{(q^4;q^4)_{k}}
=\frac{J_{2}^{3}J_{3,6}}{J_{1}^{2}J_{4}^{2}}.
\end{align}
This proves the special case $c=-1$ of \eqref{thm3.7-4}.
\end{rem}

\subsection{Concluding remarks}
A striking byproduct of our study is two new counterexmaples to Zagier's duality conjecture. Recall that Wang \cite{Wang2024} provided the first set of counterexamples to it. He showed that the Nahm sums $f_{A,B_i,1/16}(q)$ ($i=1,2$) are modular for
\begin{align}\label{exam-data-1}
A=\begin{pmatrix} 1 & 0 & 0 & -1/2 \\
0 & 1 & 0 & -1/2 \\
0 & 0 & 1 & -1/2 \\
-1/2 & -1/2 & -1/2 & 1 \end{pmatrix},  \quad B_1=\begin{pmatrix} 0 \\ 1/2 \\ 1/2 \\ -1/2 \end{pmatrix}, \quad B_2=\begin{pmatrix} 0 \\ 1/2 \\ 1/2 \\ 0\end{pmatrix}.
\end{align}
However, the dual Nahm sums $f_{A^\star,B_i^\star,C'}(q)$ ($i=1,2$) are not modular for any $C'$ where
\begin{align}\label{exam-data-2}
A^\star=\begin{pmatrix} 2 & 1 & 1 & 2 \\
1 & 2 & 1 & 2 \\
1 & 1 & 2 & 2 \\
2 & 2 & 2 & 4
\end{pmatrix}, \quad  B_1^\star=\begin{pmatrix}
0 \\ 1/2 \\ 1/2 \\ 0 \end{pmatrix}, \quad B_2^\star=\begin{pmatrix} 1 \\ 3/2 \\ 3/2 \\ 2\end{pmatrix}.
\end{align}
The reason of being nonmodular is that the dual Nahm sums can be expressed as sums of two modular forms of weights zero and one, respectively.

Motivated by and similar to Wang's discovery, during the above study of the lift of Zagier's Example 1, we considered the following matrices:
\begin{align}
A=\begin{pmatrix} 3/2 & 1/2 & 1 \\  1/2 & 3/2 & 1 \\ 1 & 1 & 2 \end{pmatrix}, \quad A^\star=\begin{pmatrix} 1 & 0 & -1/2 \\  0 & 1 & -1/2 \\ -1/2 & -1/2 & 1 \end{pmatrix}.
\end{align}
We proved that $f_{A,B_i,C}(q)$ are nonmodular for any $C$  where $B_1=(\frac{1}{2},\frac{1}{2},0)^\mathrm{T}$, $B_2=(1,1,1)^\mathrm{T}$, but  their dual Nahm sums $f_{A^\star,B_i^\star,C_i}(q)$ are modular.  Hence they serve as new counterexamples to Zagier's duality conjecture.


\begin{thebibliography}{0}
\bibitem{Andrews-book}G.E. Andrews, The Theory of Partitions, Addison--Wesley, 1976; Reissued Cambridge, 1998.

\bibitem{Andrews-Berndt} G.E. Andrews and B.C. Berndt, Ramanujan’s Lost Notebook, Part II, Springer 2009.

\bibitem{BIS} D.M. Bressoud, M.E.H. Ismail and D. Stanton, Change of base in Bailey pairs, Ramanujan J. 4 (2000), 435--453.

\bibitem{CGZ} F. Calegari, S. Garoufalidis and D. Zagier, Bloch groups, algebraic K-theory, units, and Nahm's conjecture, Ann. Sci. \'Ec. Norm. Sup\'er. (4) 56 (2023), no.\ 2, 383--426.

\bibitem{CRW} Z. Cao, H. Rosengren and L. Wang, On some double Nahm sums of Zagier, J. Combin. Theory Ser. A 202 (2024), Paper No. 105819.

\bibitem{Feigin} I. Cherednik and B. Feigin, Rogers--Ramanujan type identities and Nil-DAHA, Adv. Math. 248 (2013), 1050--1088.

\bibitem{GKPS} D. Gang, H. Kim, B. Park and S. Stubbs, Three dimensional topological field theories and Nahm sum formulas, arXiv: 2411.06081.

\bibitem{Gasper-Rahman} G. Gasper and M. Rahman, Basic Hypergeometric Series, 2nd Edition, Encyclopedia of Mathematics and Its Applications, Vol. 96, Cambridge University Press, 2004.

\bibitem{Kac} V. Kac and M. Wakimoto, Modular invariant representations of infinite dimensional Lie algebras and superalgebras, Proc. Nat. Acad. Sci. 85 (1988), 4956--4960.

\bibitem{McLaughlin} J. Mc Laughlin, Topics and methods in $q$-series, Monographs in Number Theory, 8, World Scientific Publishing Co.
Pte. Ltd., Hackensack, NJ, 2018.


\bibitem{LeeThesis} C.-H. Lee, Algebraic structures in modular $q$-hypergeometric series, PhD Thesis, University of California, Berkeley, 2012.



\bibitem{Lovejoy2004} J. Lovejoy, A Bailey lattice, Proc. Am. Math. Soc. 132 (2004), 1507--1516.

\bibitem{Rogers}L.J. Rogers, Second memoir on the expansion of certain infinite products, Proc. London Math. Soc. 25 (1894), 318--343.

\bibitem{Slater}L.J. Slater, Further identities of the Rogers--Ramanujan type, Proc. Lond. Math. Soc. (2) 54 (1) (1952), 147--167.

\bibitem{VZ}  M. Vlasenko and S. Zwegers, Nahm's conjecture: asymptotic computations and counterexamples, Commu. Number Theory Phys. 5(3) (2011), 617--642.

\bibitem{XiaYao} E.X.W. Xia and O.X.M. Yao, Analogues of Ramanujan's partition identities,
Ramanujan J. 31 (2013), 373--396.

\bibitem{Wang-rank2} L. Wang, Identities on Zagier's rank two examples for Nahm's problem, Res.\ Math.\ Sci. (2024) 11:49.

\bibitem{Wang-rank3} L. Wang, Explicit forms and proofs of Zagier's rank three examples for Nahm's problem, Adv. Math. 450 (2024), 109743.

\bibitem{Wang2024} L. Wang, Counterexamples to Zagier's duality conjecture on Nahm sums, arXiv:2411.09701v3.

\bibitem{Zagier} D. Zagier, The dilogarithm function, in Frontiers in Number Theory, Physics and Geometry, II, Springer, 2007, 3--65.

\bibitem{ZwegersTalk} S. Zwegers,  presentation. In: Workshop on Mock Modular Forms in Combinatorics and Arithmetic Geometry, American Institute of Mathematics, Palo Alto, California Mar. 8--12. 2010.

\end{thebibliography}
\end{document}